\documentclass[	fontsize=12pt, 
	BCOR=12mm, 
	DIV=calc, 
	headinclude,
	bibliography=oldstyle,
	bibliography=totoc]{scrbook}

\usepackage{amsthm,amsfonts,amsmath, amssymb, mathtools, bbm}
\usepackage[mathscr]{eucal}

\KOMAoptions{DIV=calc}

\usepackage{fourier} 
\usepackage[scaled=0.875]{helvet} 

\usepackage{graphicx}
\usepackage{enumerate}
\usepackage{scrpage2}
\setheadsepline{.1pt}
\newtheorem{theorem}{Theorem}[section]
\newtheorem{lemma}[theorem]{Lemma}
\newtheorem{proposition}[theorem]{Proposition}
\newtheorem{corollary}[theorem]{Corollary}
\newtheorem*{theorem*}{Theorem}

\newtheorem*{conjecture}{Conjecture}

\theoremstyle{definition}
\newtheorem{definition}[theorem]{Definition} 
\newtheorem{question}[theorem]{Question} 

\newtheorem{eg}[theorem]{Example}
\newtheorem{remark}[theorem]{Remark}

\newcommand{\ee}[1]{\mathbb{E}\left[ #1 \right]} 
\newcommand{\cee}[2]{\mathbb{E}\left[\left. #1 \right| #2 \right]} 
\newcommand{\var}[1]{Var\left( #1 \right)}

\DeclareMathOperator{\Tr}{Tr}
\mathtoolsset{showonlyrefs=true}

\theoremstyle{plain}

  {\end{enumerate}}

\newcommand{\no}{\noindent}

\usepackage{xcolor}	 
\definecolor{mycolour}{RGB}{150,0,0}

\hypersetup{
  colorlinks,
  citecolor=mycolour,
  linkcolor=mycolour,
  urlcolor=mycolour}

\begin{document}

\title{On Critical Points of Random Polynomials and Spectrum of Certain Products of Random Matrices.}
\author{Tulasi Ram Reddy A}
\begin{titlepage}
	\begin{center}
		\vspace{.5in}
	\begin{Large}
		\textbf{On Critical Points of Random Polynomials and}\\
		 \textbf{ Spectrum of Certain Products of Random Matrices}\\
	\end{Large}
		\vspace{0.8in}
		{\large  A Dissertation \\ 
			submitted in partial fulfilment  \\
		\vspace{0.04in}
	 of the requirements for the award of the\\ 
	 	\vspace{0.04in}
	 	degree of} \\ 		
		\vspace{0.2in}
	{{\Large
		\textbf{Doctor of Philosophy}}}\\  
		\vspace{.05in}
		{\large by}\\
		\vspace{0.05in}
		{\large Tulasi Ram Reddy Annapareddy}\\
		\includegraphics[scale=.2]{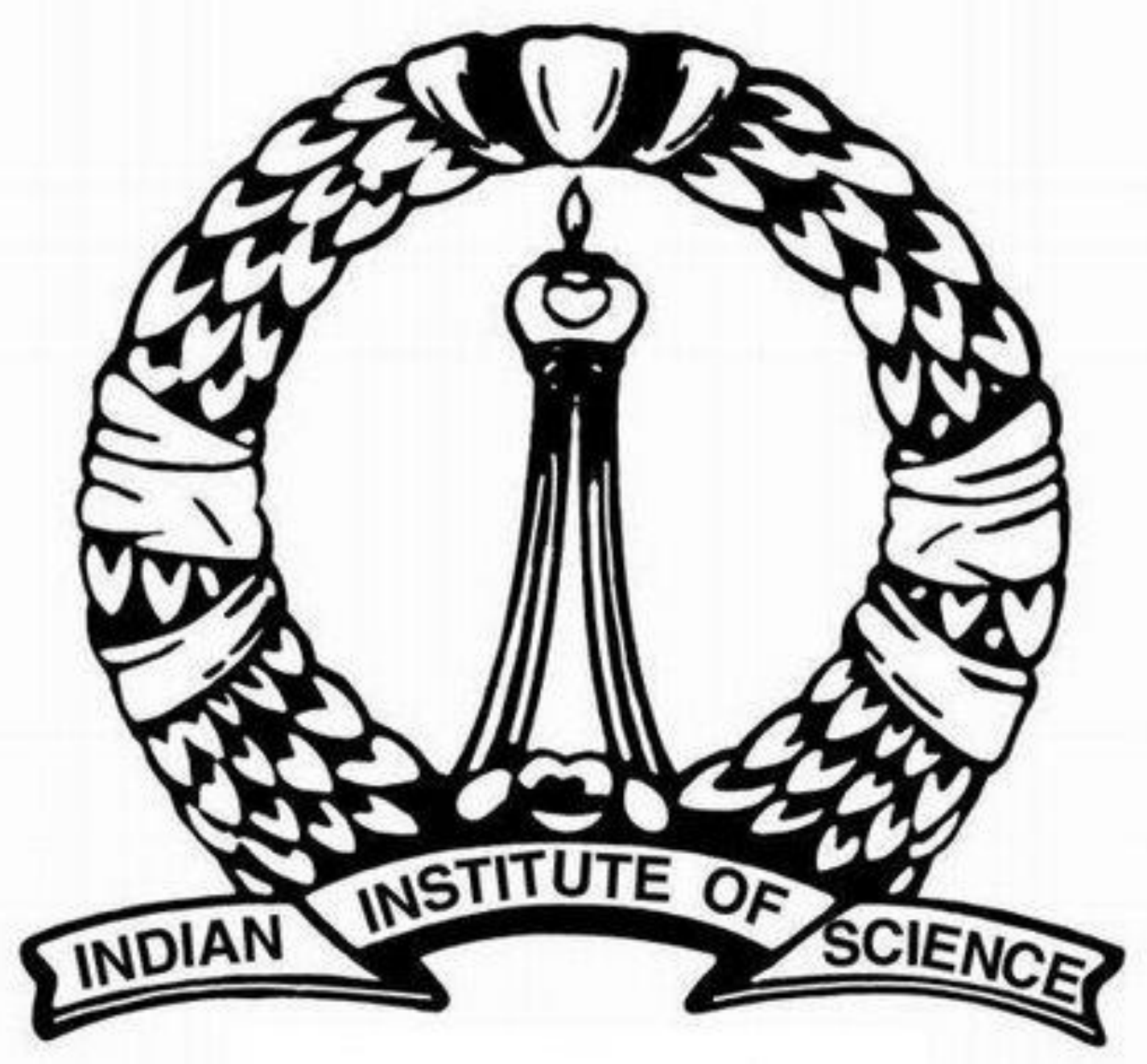} \\
		{\large Department of Mathematics \\
		Indian Institute of Science \\
			Bangalore - 560012 \\
		July 2015 \\}
	\end{center}
\end{titlepage}

\frontmatter
\chapter{Declaration}
\vspace{0.5in}
\noindent I hereby declare that the work reported in this thesis is entirely original and has
been carried out by me under the supervision of Prof.~Manjunath Krishnapur at the Department of
Mathematics, Indian Institute of Science, Bangalore. I further declare
that this work has not been the basis for the award of any degree, diploma, fellowship,
associateship or similar title of any University or Institution.\\
\vspace*{1in}

\noindent $\begin{array}{lcr}
\textrm{Tulasi Ram Reddy A} & \hspace*{1.95in} &~ \\
\textrm{S. R. No. 6910-110-101-08085} & \hspace*{1.95in}   &~ \\
\textrm{Indian Institute of Science} & \hspace*{1.95in}   & ~ \\
\textrm{Bangalore} & \hspace*{1.95in}  & ~ \\
~ & \hspace*{1.95in}   & ~ \\
~ & \hspace*{1.95in}   & ~ \\
~ & \hspace*{1.95in}   & ~ \\
~& \hspace*{1.95in}   & \textrm{Prof.~Manjunath Krishnapur} \\
~& \hspace*{1.95in}  & \textrm{(Research advisor)}
\end{array}$

\frontmatter

\clearpage
\thispagestyle{empty}
\par\vspace*{.35\textheight}{\centering TO \\[2em]
\large\it My Parents\\
and\\
\large\it Teachers \par}

\chapter{Acknowledgements}

For an outsider a doctoral thesis might appear as a solitary endeavour. But several people (including virtual communities) have contributed in various forms to this thesis. I make an attempt here to thank the people who have contributed to the same. The set of people I have acknowledged here is a subset of the many people who have helped me in this endeavour and is far from complete.
 
 I am deeply indebted to my thesis advisor Manjunath Krishnapur. It is a privilege to be his student. To say the least he has played several roles from being my thesis advisor to a great companion. I believe that I have also acquired some of his inexhaustible enthusiasm towards mathematics, which is indeed very contagious.
 
 My interest in the study of random polynomials surged after discussions with Zakhar Kabluchko during the Trondheim spring school 2013, along with several conversations with Manjunath. 
 
  I have immensely benefited by the courses offered in the Mathematics department at IISc. The many thought provoking talks and seminars held at IISc, ISI, ICTS and TIFR-CAM have also been helpful.

Probabilists in Bangalore deserve a special mention for their efforts in organising various academic activities in the city. They, along with many probabilists visiting Bangalore, have certainly made this city very lively and thereby stimulating my research activity at various levels.

I was introduced to mathematics at the Indian Statistical Institute, Kolkata. Courses taught by BV Rao, late SC Bagchi, Arup Bose, Gopal Basak, Arnab Chakraborthy and Probal Chaudhuri among others  created first impressions on mathematics and probability in me. Arni Srinivasa Rao encouraged me to pursue research in mathematics and has been great support since. From my high school teachers I recall K Subba Rao and M Radhakrishna who constantly encouraged me to do mathematics. 

I take this opportunity to thank various organizations for providing their support generously. I thank NBHM for providing me travel grant to attend the School on Random matrices and Growth models in 2013 in Trieste, Italy. Apart from this, I was also supported by NBHM during my first year of Ph.D program. I thank CSIR for supporting me with the SPM fellowship. I am grateful to KVPY which provided me a fellowship during my undergraduate days. These fellowships have helped me make choices without any confusion at various junctures in my life.

My friends at IISc have bestowed me with several moments which can be cherished forever. This wouldn't have been possible without people like Arpan, Divakaran, Jaikrishnan, Kartick, Nanda, Pranav, Prathamesh, Rajeev, Sayani and Vikram. 

Kartick, Manjunath and Nanda read through this thesis carefully before pointing out many mistakes. They also suggested many invaluable changes. A special thanks to them. I would also like to thank them along with Koushik Saha for having good discussions in the subject.

I thank the office staff in the Mathematics department and the many workers at IISc for ensuring an environment which facilitated a very smooth stay for me in the past five years.

Many of my endeavours have posed several challenges for my parents at times. Regardless, they have endorsed my decisions and backed me unflinchingly all along. No less was the patience and love unveiled by my sister and brother-in-law. My grandfather, who probably is my first friend ever, used to present creative answers to my questions during our rounds in the fields. He indeed deserves a special mention here. My aunts, uncles and cousins have made my visits to home very eventful and something to crave for. Friends were constantly in touch with me, despite no efforts from my side in reaching to them. I am sure all of them will be happy seeing this thesis.

\chapter{Abstract}
	In the first part of this thesis, we study critical points of random polynomials. We
	choose two
	deterministic sequences of complex numbers, whose empirical measures converge to the
	same probability measure in complex plane. We make a sequence of polynomials whose
	zeros are chosen from either of sequences at random. We show that the limiting
	empirical
	measure of zeros and critical points agree for these polynomials. As a consequence we
	show that when we randomly perturb the zeros of a deterministic sequence of
	polynomials,
	the limiting empirical measures of zeros and critical points agree. This result can be
	interpreted as an extension of earlier results where randomness is reduced. Pemantle
	and
	Rivin initiated the study of critical points of random polynomials. Kabluchko proved
	the result
	considering the zeros to be i.i.d. random variables.\\

		  In the second part we deal with the spectrum of products of Ginibre matrices. Exact eigenvalue density is known for a very few matrix ensembles. For the known ones they often lead to determinantal point process. Let $X_1,X_2,\dots, X_k$ be i.i.d Ginibre matrices of size $n \times n$ whose entries are standard complex Gaussian random variables. We derive eigenvalue density for matrices of the form $X_1^{\epsilon_1}X_2^{\epsilon_2}\dots X_k^{\epsilon_k}$, where $\epsilon_i=\pm1$ for $i=1,2,\dots,k$. We show that the eigenvalues form a determinantal point process. The case where $k=2$, $\epsilon_1+\epsilon_2=0$ was derived earlier by Krishnapur. In the case where $\epsilon_i=1$ for $i=1,2,\dots,n$ was derived by Akemann and Burda. These two known cases can be obtained as special cases of our result.

\tableofcontents

\mainmatter

\chapter{Introduction}
\label{ch:introduction1}
The fundamental theorem of algebra states that every polynomial of degree $n$ of a single variable has, counted with multiplicity, exactly $n$ zeros (roots) in the complex plane.  Abel and Galois proved that the roots of any polynomial of degree $5$ or higher cannot be expressed in terms of radicals. Like in many other problems, where exact formulae are not known (or the formulae not amenable to analysis), one may study the statistics of a `typical' polynomial. This is done by equipping a probability measure on the space of polynomials (or by introducing randomness within the polynomial) and choosing a random polynomial according to this measure. There are two natural ways of inducing randomness into polynomials - one by considering the characteristic polynomials of a random matrix and other by choosing the coefficients of polynomials to be random variables. Both cases are of interest. The former leads to the study of random matrices while the latter to the theory of random polynomials.

In the theory of random polynomials, the central problem is to understand the behaviour of zeros which is studied by choosing coefficients to be random variables (often independent). The study of random polynomials many times provide us with interesting phenomenon. For example, the zeros of Kac's polynomials, which are defined as random polynomials with i.i.d. complex Gaussian random variables as coefficients, accumulate near the unit circle. Pemantle and Rivin, in \cite{pemantle}, considered a sequence of polynomials whose zeros are i.i.d. complex random variables. They conjectured that the empirical measures of zeros and critical points of these polynomials agree in limit. 

The first part of this thesis is inspired by Kabluchko's proof \cite{kabluchko} to the problem of Pemantle and Rivin. Here we construct a random sequence of zeros by choosing its terms from two predefined deterministic sequences. It is shown that for the sequence of random polynomials, whose zeros are the terms of the random sequence constructed above, the limiting empirical measures of zeros and critical points agree. This phenomenon fails in general for deterministic sequence of polynomials. Examples of deterministic sequence of polynomials where the limiting measures of zeros and critical points don't agree can be constructed. However, as a consequence of the previous result, it can be shown that if we slightly perturb the zeros of these polynomials, then the limiting measures of zeros and critical points will agree.  

Hannay in \cite{hannay}  and Hanin in \cite{hanin1} observed that the critical points of random polynomials are closely paired with the zeros of the polynomials. In this setting, we study the matching distance between zeros and critical points. When the zeros of the polynomials are all i.i.d. real valued random variables, it is shown that the matching distance between zeros and critical points of these polynomials will remain bounded if the random variables have finite first moment. It is also shown that in limit the bound is exactly the first moment of these random variables.



The theory of random matrices deals with the study of eigenvalues of a random matrix. Finding exact eigenvalue density is an important problem in random matrix theory. There are only a handful matrix ensembles for which the exact eigenvalue density is known. For these ensembles, it is often true that the eigenvalues constitute an important class of point processes called determinantal point processes for which a theory and framework is already available for analysis. Of the known ensembles, very few are non-hermitian matrix ensembles. Ginibre, in \cite{ginibre}, derived the eigenvalue density for a matrix whose entries are i.i.d. complex Gaussian random variables. These matrices are called Ginibre matrices since. Krishnapur, in \cite{manjunath} derived the eigenvalue density for $A^{-1}B$, where $A$ and $B$ are independent Ginibre matrices. Akemann and Burda in \cite{akemann} derived the eigenvalue density for random matrices obtained as the product of independent Ginibre matrices. A generalization for these matrices is to consider product of independent matrices where each matrix or its inverse is a Ginibre matrix. In this thesis, the eigenvalue density for these matrices is derived and it is shown that they form a determinantal point process. This result generalizes  all the previously mentioned results.

Studying the behaviour of real eigenvalues for real random matrices and real zeros for real random polynomials have posed different challenges (due to the lack of conventional symmetries) and simultaneously offered various insights. A different problem on the products of i.i.d. real matrices of fixed size with i.i.d. entries was considered by Lakshminarayan in \cite{arul}. He considered the case when the entries are i.i.d. real Gaussian random variables. He conjectured that the probability of the product of these matrices have all real eigenvalues, converge to $1$ as the size of the product increase to infinity. He established this conjecture for the matrices of size $2 \times 2$. Forrester, in \cite{forrester}, proved this conjecture for any $k\geq1$. It is natural to believe that this phenomenon is universal and hence may hold for any matrix with i.i.d. entries.  We show this in a case where the entries of these matrices are distributed according to the probability measure $\mu$ which has an atom. 

\section{Outline}
We now outline the contents of this thesis briefly. This thesis broadly deals with two themes.  In the first part we study the zeros and critical points of random polynomials, which is covered in Chapters 2, 3 and 4.

\begin{itemize}
\item
In Chapter 2, a brief history of the results relating critical points and zeros of random polynomials are given. We deal with sequences of deterministic polynomials to provide explicit examples in which the limiting measures of zeros and critical points do not agree. Thereafter, a little randomness is introduced into these polynomials which ensures that the limiting measures agree. We also discuss some of their consequences. 
\item
In Chapter 3, we prove that the limiting measure of zeros and critical points of a sequence of random polynomials agree when a little randomness is introduced.  The results stated in Chapter 2 are proved.
\item
In Chapter 4, we consider the matching problem between zeros and critical points of random polynomials.  We show that when the zeros are real and i.i.d. from a given distribution with finite first moment, then the $\ell^1$ matching distance will be finite. We also consider the spacing between the zeros and critical points of the random polynomials. In the case where zeros are i.i.d. exp($\lambda$) random variables, we show that the extremal critical point is much closer to the extremal zero of the random polynomial.

\end{itemize}

In the second part we study the eigenvalues of certain products of random matrices. This is covered in Chapters 5 and 6.

\begin{itemize}
\item
In Chapter 5, we derive the exact eigenvalue density for $X=X_1^{\epsilon_1}X_2^{\epsilon_2}\dots X_n^{\epsilon_n}$, where $\epsilon_i=\pm1$ and $X_i$ are i.i.d. complex Ginibre matrices. In other words, we derive the eigenvalue density for products of complex Ginibre matrices of fixed size in which some of them are inverted. It is also observed that they form a determinantal point process.
\item
In Chapter 7, we present a stronger version of the conjecture, by Lakshminarayan in \cite{arul}. We prove the conjecture in a special case when the entries of the matrices are distributed according to $\mu$ which has an atom.

\end{itemize}

%
%
%
%
%
%
%
%

\chapter{Critical points of random polynomials}
\label{ch:criticalpoints4}
\section{Introduction}In this chapter we will investigate the distribution of the critical points in relation to the zeros of a polynomial. For a holomorphic function $f:\mathbb{C} \rightarrow \mathbb{C}$ a point   $z \in \mathbb{C}$ is called a critical point of $f$ if $f'(z)=0$.


The oldest known result relating the  zeros and critical points of a polynomial is Gauss-Lucas theorem, which states that the critical points of any polynomial with complex coefficients lie inside the convex hull formed by the zeros of the polynomial.

\begin{theorem}[Gauss-Lucas;  See Chapter 2, Theorem 6.1 in~\cite{marden}]\label{guass-lucas}
	Let $P$ be a non-constant complex polynomial then the zeros of $P'$ are contained in the convex hull formed by the zeros of $P$.
\end{theorem}

In general nothing more can be said. Our interest is in dealing with sequences of polynomials, usually randomness included, with increasing degrees. We consider the case in which the point cloud made from the zeros of these polynomials will converge to a probability measure in the complex plane. We want to understand the behaviour of critical points of these polynomials. We recall the definition of weak convergence.

\begin{definition}
	For a sequence of probability measures, $\{\mu_n\} \text{ and } \mu$ on $\text{$\mathbb{C}$}$, we say that $\mu_n \xrightarrow{w} \mu$ \textit{weakly}, if  for any $ f\in C_c^{\infty}(\text{$\mathbb{C}$})$, we have $\lim\limits_{n \rightarrow \infty}\int_{X}^{}fd\mu_n = \int_{X}^{}fd\mu$. 
\end{definition}
The following definition formalizes the notion of point cloud converging to a probability measure. Here the point cloud being the collection of terms from a sequence of complex numbers.
\begin{definition}We say a sequence of complex numbers $\{a_n\}_{n\geq1}$ to be \textit{$\mu$-distributed} if its empirical measures $\frac{1}{n}\sum_{k=1}^{n}\delta_{a_k}$ converge weakly to the probability measure $\mu$.
\end{definition}

In the next section we study the critical points of deterministic sequence of polynomials. We show that if all the zeros are confined in regions that are well separated, then the critical points also confine to these regions. Then,  in subsequent sections we discuss a few examples in which the limiting measures of zeros and of critical points do not agree. 
Later we give a brief overview of existing results in the literature where a sequence of random polynomials are considered. In all these cases it was shown that the limiting measures of zeros and critical points agree.

 In the Section 2.3 we show the results we have obtained and discuss their consequences. In the first result we construct the zeros sequence, for each term  choose the term from one of the two deterministic sequences at random. We construct a sequence of polynomials whose zeros are the terms of this sequence. We show that the limiting empirical measures of zeros and critical points of these polynomials agree. We then state the corollaries of this result, where we perturb this sequence randomly and show that the limiting measure of zeros and critical points agree. In our second result we consider a random rational function which can be used to get a generalized derivative and show that the limiting distribution of  zeros and poles agree. As a corollary of this we show that if we choose a random subsequence of a deterministic sequence then the limiting measures of zeros and critical points agree. In the last section we prove the corollaries mentioned earlier. We defer the proofs of the theorems to the next chapter. 

We will recall a well known proof of Gauss-Lucas theorem. If $z_1,z_2,\dots,z_n$ are the roots of the polynomial $P$, then for some $c$, $P(z)=c(z-z_1)(z-z_2)\dots(z-z_n)$. Define $L(z):=\frac{P'(z)}{P(z)}=\sum_{k=1}^{n}\frac{1}{z-z_k}$. If $z$ is a zero of $P'$ and not equal to any of the $z_i$s, then $L(z)=0$. Hence,

\[
\sum_{k=1}^{n}\frac{1}{z-z_k}=0 \text{\hspace{12 pt}  or, \hspace{12 pt} } \sum_{k=1}^{n}\frac{\overline{z}-\overline{z}_k}{|z-z_k|^2}=0.
\]
Therefore if $L(z)=0$, then $z$ satisfies,
\[
z=\frac{\sum_{k=1}^{n}\frac{1}{|z-z_k|^2}z_k}{\sum_{k=1}^{n}\frac{1}{|z-z_k|^2}}.
\]
In the above equation $z$ is expressed as a convex combination of $z_k$s. In the other case where $z$ is one of the $z_k$s, it trivially true that $z$ is in the convex hull formed by $z_1,z_2,\dots,z_n$.

For a different proof of this theorem the reader can refer to \cite[Chapter-2, Theorem 6.1]{marden}. After the proof of Gauss-Lucas theorem, there have been several results concerning  critical points and zeros of  polynomials. Interested reader may see the references in \cite{pemantle}. Several conjectures on the same can be found in \cite{mardenconjectures} and \cite{borcea}. A brief survey on results connecting zeros and critical points of polynomials can be found in \cite{sury}.
In the proof of Gauss-Lucas theorem we have defined a function $L(z)=\sum_{k=1}^{n}\frac{1}{z-z_k}$. It is interpreted as potential at the point $z$ due to unit charges present at points $z_1,z_2,\dots,z_n$. In studying critical points this potential function plays a key role. All the results in this chapter are obtained by analyzing this function.

\section{Critical points of a sequence of deterministic polynomials.}

Consider  a sequence of deterministic polynomials. In the case where all the zeros of these polynomials are in a bounded convex set, Gauss-Lucas theorem asserts that the critical points of these polynomials lie inside the same set. In this context we state a well known related result by Walsh.
\begin{theorem}[J.L.Walsh ~\cite{walsh}]
Let $C_1$, $C_2$ be disks with centres $c_1$, $c_2$ and radii $r_1$, $r_2$. Let $P$ be a polynomial of degree $n$ with all its zeros in $C_1\cup C_2$, say $n_1$ zeros in $C_1$ and $n_2$ zeros in $C_2$. Then $P$ has all its critical points in $C_1\cup C_2 \cup C_3$, where $C_3$ is the disk with centre $c_3$ and radius $r_3$ given by
\[
c_3=\frac{n_1c_2+n_2c_1}{n}, \hspace{5 pt} r_3=\frac{n_1r_2+n_2r_1}{n}.
\]
Furthermore, if $C_1$, $C_2$ and $C_3$ are pairwise disjoint, then $C_1$ contains $n_1-1$ critical points, $C_2$ contains $n_2-1$ critical points and $C_3$ contains $1$ critical point.
\end{theorem}

Inspired by the above theorem of Walsh, we derive the following result. Consider a sequence of polynomials whose zeros are in well separated clusters (say for example separated unit disks). Further assume that the number of zeros in these sets grow proportionately. Under these assumptions we show that in a neighbourhood of each of these sets the number of critical points differ from the number of zeros by at most a constant number.

\begin{theorem}\label{walsh_general}
	Let $S_1,S_2,\dots, S_k$ be pairwise disjoint bounded convex sets in the complex plane. Assume that $\mbox{diam}(S_i)\leq 1$,  for $i=1,2,\dots,k$ and the distance of separation between any two $S_i,S_j$, for $i\neq j$, is at least $5k$. Define the sequence of polynomials $\{P_n\}_{n\geq 1}$ as $P_n(z):=\prod\limits_{j=1}^{k}\prod\limits_{i=1}^{n}(z-z^{(j)}_i)$, where  $z_i^{(j)} \in S_j$ for $j=1,\dots, k$ and $i=1,\dots, n$. Then, for  $\epsilon > \frac{3-\sqrt{5}}{2}$ we have a constant $c(k,\epsilon)$, such that the number of zeros of $P_n'(z)$ in the $S_i^{\epsilon}$ is at least $n-c(k,\epsilon)$ for any $i=1,\dots,k$. 
\end{theorem}

\begin{proof}

We will estimate the number of critical points of $P_n(z)$ in  $S_1^\epsilon$ the $\epsilon$ neighbourhood of $S_1$. Let the diameter of the sets $S_i$ be at most $d$ for any $i \in \{1,2,\dots,k\}$. Assume that the separation between any two $S_i,S_j$ is at least $d_s$, where $d_s>0$. In the course of this proof we will substitute the values $d=1$ and $d_s=5k$ as given in the statement of the theorem.  For $\epsilon$ small enough, we know that there are $n$ zeros of $P_n(z)$ in $S_1^{\epsilon}$. Argument principle computes the difference between zeros and poles of a meromorphic function in a domain by evaluating a certain integral on the boundary of the domain. We will use argument principle to estimate the critical points of $P_n(z)$ in $S_1^{\epsilon}$. A version of argument principle is stated below.

Let $f: U \rightarrow\mathbb{C}$ be a meromorphic function on a simply connected domain $U$ and let $C$ be a rectifiable simple closed curve in $U$. Assume that $f$ does not vanish on $C$. Then 
\begin{equation}
\frac{1}{2\pi i}\oint_{C}\frac{f'(z)}{f(z)}dz = N_Z\left(f,C\right)-N_P\left(f,C\right)\label{argumentprinciple}
\end{equation}  
where $N_Z\left(f,C\right)$ and $N_P\left(f,C\right)$ are the number of zeros and poles of $f$ enclosed by the curve $C$. 

Define $L_n(z):=\frac{P_n'(z)}{P_n(z)}=\sum\limits_{j=1}^{k}\sum\limits_{s=1}^{n}\frac{1}{z-z_s^{(j)}}$ and notice that the zeros and poles of $L_n(z)$ are zeros of $P_n'(z)$ and $P_n(z)$ respectively. We shall apply the argument principle for the function $L_n(z)$ for the boundary curve $\gamma_1^\epsilon$ obtained from $\partial S_1^{\epsilon}$. Assume that there are no zeros of $P_n'(z)$ on the curve $\gamma_1^\epsilon$. 
Hence applying the formula \eqref{argumentprinciple} to $L_n(z)$ we get,

\begin{align}
|N_Z\left(L_n,\gamma_1^\epsilon\right)-N_P\left(L_n,\gamma_1^\epsilon\right)| & =  \bigg|\frac{1}{2\pi}\oint\limits_{\gamma_1^\epsilon}\frac{L_n'(z)}{L_n(z)}dz\bigg| & \text{(or)} \\
|N_Z\left(P_n,\gamma_1^\epsilon\right)-n| & \leq \frac{1}{2\pi}\oint\limits_{\gamma_1^\epsilon}\bigg|\frac{L_n'(z)}{L_n(z)}\bigg||dz| \label{lineintegral}
\end{align}
For the above integrand we will give an upper bound for the numerator and a lower bound for the denominator on the curve $\gamma_1^\epsilon$. Because $|z-z_s^{(j)}|\geq\epsilon$ and from the triangle inequality, for $z$ on the curve $\gamma_1^\epsilon$ we get,
\begin{equation}
|L_n'(z)|=\Biggl|-\sum\limits_{j=1}^{k}\sum\limits_{s=1}^{n}\frac{1}{\left(z-z^{(j)}_s\right)^2}\Biggr| \leq \sum\limits_{j=1}^{k}\sum\limits_{s=1}^{n}\frac{1}{\bigl|z-z^{(j)}_s\bigr|^2} \leq \frac{kn}{\epsilon^2}.\label{numerator}
\end{equation}
Similarly for $L_n(z)$ using triangle inequality we get,
\begin{align}
|L_n(z)|=\left|\sum\limits_{j=1}^{k}\sum\limits_{s=1}^{n}\frac{1}{z-z^{(j)}_s}\right| & \geq \left|\sum\limits_{s=1}^{n}\frac{1}{z-z^{(1)}_s}\right|- \left|\sum\limits_{j=2}^{k}\sum\limits_{s=1}^{n}\frac{1}{z-z^{(j)}_s}\right|\label{eqn:lemma0:1} \\
\end{align}
The first term in the right most expression in \eqref{eqn:lemma0:1} is invariant under multiplication by $e^{i\theta}$. Therefore for any $\theta \in [0,2\pi)$ we have,

\begin{align}
|L_n(z)| &\geq\left|\sum\limits_{s=1}^{n}\frac{e^{i\theta}}{z-z^{(1)}_s}\right|-\left|\sum\limits_{j=2}^{k}\sum\limits_{s=1}^{n}\frac{1}{z-z^{(j)}_s}\right|,\\
&\geq \left|\sum\limits_{s=1}^{n}\Im\left(\frac{e^{i\theta}}{z-z^{(1)}_s}\right)\right|- \left|\sum\limits_{j=2}^{k}\sum\limits_{s=1}^{n}\frac{1}{z-z^{(j)}_s}\right|, \\ & \geq \frac{n\epsilon}{(d+\epsilon)^2}-\frac{(k-1)n}{d_s-\epsilon}.
\end{align}
because $S_1^\epsilon$ is a convex set, there is a line passing through $z\in \partial S_1^\epsilon$ (from separating hyperplane theorem) such that all the points $z_s^{(1)}$ lie on the same side of this line. Let $\theta$  be the angle made by this line with real axis, then all the terms in the $\sum\limits_{s=1}^{n}\Im{\frac{e^{i\theta}\left(\overline z- \overline z_s^{(1)}\right)}{\left|z-z_s^{(1)}\right|^2}}$ have the same sign and the absolute value of the numerators is atleast $\epsilon$ and denominators with at most $(d+\epsilon)^2$. Similarly $\bigg|\sum\limits_{j =2}^{k}\sum\limits_{s=1}^{n}\frac{1}{z-z_s^{(j)}}\bigg| \leq \sum\limits_{j=2}^{k}\sum\limits_{s=1}^{n}\frac{1}{|z-z_s^{(j)}|}$, and for $j\neq 1$ the denominator in the previous expression $|z-z_s^{(j)}|$ is at least $d_s-\epsilon$, because $|z-z_s^{(j)}| \geq d(S_1^\epsilon,S_j)\geq d(S_1,S_j)-\epsilon \geq d_s-\epsilon$. For $z \in \partial S_1^\epsilon$, we obtain
\begin{equation}
\left|L_n(z)\right| \geq \frac{n\epsilon}{(d+\epsilon)^2}-\frac{(k-1)n}{d_s-\epsilon}. \label{denominator}
\end{equation}

By substituting $d=1$, for the right hand side of \eqref{denominator} to be positive, we need $d_s$ to satisfy 
\begin{equation}
 d_s>\epsilon+\frac{(k-1)(1+\epsilon)^2}{\epsilon}.\label{eqn:d_sandepsilon}
\end{equation}
For any choice of $d_s\geq 5k-2$ and $\epsilon \in [\frac{3-\sqrt{5}}{2},\frac{3+\sqrt{5}}{2}]$, the inequality \eqref{eqn:d_sandepsilon} is satisfied. For these choices of variables we have,
\begin{align}
|N_Z\left(P_n,\gamma_1^\epsilon\right)-n| & \leq \frac{1}{2\pi}\oint\limits_{\gamma_1^\epsilon}\bigg|\frac{L_n'(z)}{L_n(z)}\bigg||dz|,\\
& \leq
\frac{1}{2\pi}\oint\limits_{\gamma_1^\epsilon}\dfrac{\frac{kn}{\epsilon^2}}{\frac{n\epsilon}{(1+\epsilon)^2}-\frac{(k-1)n}{d_s-\epsilon}}|dz|,\\
&\leq \frac{1+2\epsilon}{2\epsilon^2}\dfrac{k}{\frac{\epsilon}{(1+\epsilon)^2}-\frac{(k-1)}{d_s-\epsilon}} =: c(k,\epsilon,d_s). \label{eqn:integral}
\end{align}

The inequality \eqref{eqn:integral} is obtained from the fact  $p\leq \pi d$, where $p$ is the perimeter of the convex set and $d$ is the diameter of the convex set. Substituting $d_s=5k$ we obtain that $|N_Z(L_n,\gamma_1^\epsilon)-N_P(L_n,\gamma_1^\epsilon)|\leq c(k,\epsilon)$. By the choice of $\epsilon$, we have proved the Theorem for $\epsilon \in [\frac{3-\sqrt{5}}{2},\frac{3+\sqrt{5}}{2}]$. For $\epsilon>\frac{3+\sqrt{5}}{2}$, $S_1^\epsilon$ contains $\frac{3+\sqrt{5}}{2}$-neighbourhood of $S_1$, hence number of critical points in $S_1^\epsilon$ is at least $n-c(k,\frac{3+\sqrt{5}}{2})$. Choosing $c(k,\epsilon)=c(k,\frac{3+\sqrt{5}}{2})$, the Theorem is proved for $\epsilon>\frac{3+\sqrt{5}}{2}$.
\end{proof}

\begin{remark}
In Theorem \ref{walsh_general} we have assumed all $S_i$ have equal number of zeros of $P_n$. Instead we may assume that the number of zeros $P_n$ in $S_i$ be  $c_in$ for constants $c_1,c_2,\dots,c_k$ and obtain a similar result. The same ideas in the above proof can be used to prove this result.
\end{remark}

It was raised by Pemantle and Rivin whether it is true that if the limiting measure of zeros of the sequence of polynomials converging to a probability measure $\mu$, then the limiting measure of critical points also converge to $\mu$. We will see in the forthcoming example that this is indeed false. For convenience we will introduce the following notation.  For any polynomial $P$, denote $Z(P)$ to be the multi-set of zeros of $P$ and $\text{$\mathscr{M}$}(P)$ to be the uniform probability measure on $Z(P)$.

Here we will construct sequence of polynomials for which the limiting measure of zeros and critical points do not agree. The most commonly quoted~\cite{pemantle} sequence of polynomials in this regard is $P_n(z)=z^n-1$. In this case the limiting zero measure is the uniform probability measure on $S^1$ and the limiting critical point measure is the Dirac measure at origin. We generalize the above stated example and construct new set of examples for which the limiting measures of zeros and critical points are different.

\begin{eg}
Observe that if a polynomial has all zeros real, then all its critical points have to be real and are interlaced between the zeros of the polynomial. Consider the polynomial $P_n(z)=(z-a_1^n)(z-a_2^n)\dots(z-a_k^n)$, where $a_1,a_2,\dots,a_k$ are real numbers such that $0 < a_1 < a_2 < \dots < a_k$. Define the sequence of polynomials to be $Q_n(z)=P_n(z^n)$, then $Q_n'(z)=nz^{n-1}P_n'(z^n)$. The zero set of $Q_n$ is \[Z(Q_n)=\bigcup\limits_{j=1}^{k}\bigcup\limits_{\ell=1}^{n}\{a_je^{ 2\pi i\frac{\ell}{n}}\}.\] Where as the zero set of $Q_n'$ is \[Z(Q_n')=\left(\bigcup\limits_{j=1}^{k-1}\bigcup\limits_{\ell=1}^{n} \{b_{j,n}^{\frac{1}{n}}e^{ 2\pi i\frac{\ell}{n}}\}\right)\bigcup\{0,0,\dots,0\},\] where $b_{1,n},b_{2,n},\dots,b_{k-1,n}$ are the zeros of the polynomial $P_n'(z)$. The probability measure $\text{$\mathscr{M}$}(Q_n')$ has mass $\frac{n-1}{kn-1}$ at $0$, hence its limiting measure will have mass $\frac{1}{k}$ at $0$. On the other hand the probability measure $\text{$\mathscr{M}$}(Q_n)$ is supported on $\bigcup\limits_{j=1}^{k}a_jS^1$. Hence the limiting measures do not agree.
\end{eg}
\begin{eg}
For the second class of examples choose a sequence of complex numbers all containing in a disk of radius $r$ around 0. Let the sequence be $\{a_n\}_{n\geq1}$ and $|a_n|\leq r$. Make a sequence of polynomials using the terms of this sequence as its zeros. Define $P_n(z)=(z-a_1)(z-a_2)\dots(z-a_n)$. Using these define the polynomials $Q_n(z)=\int_{0}^{z}P_n(w)dw+(2r+d)^{n+1}$, where $d>0$. Notice that $Q_n'(z)=P_n(z)$. We will show that the polynomial $Q_n(z)$ does not vanish in the disk $\textbf{$\mathbb{D}$}_r$. For this observe

\begin{align}
\min\limits_{z \in \text{$\mathbb{D}$}_r}|Q_n(z)| & \geq |(2r+d)^n-\max\limits_{z \in \text{$\mathbb{D}$}_r}|\int_{0}^{z}P_n(w)dw||,\\ 
& \geq |(2r+d)^{n+1}-r\max\limits_{z\in \text{$\mathbb{D}$}_r}|P_n(z)||,\\
& =|(2r+d)^{n+1}-r(2r)^n| >0.
\end{align}
Therefore $Q_n(z)$ does not vanish in the disk $\text{$\mathbb{D}$}_r$. Hence for all large $n$ the zeros of $Q_n(z)$ are outside the disk $\text{$\mathbb{D}$}_{2r+d}$. Assuming that $Q_n(z)$ has a limiting zero measure, the support of the limiting zero measure of $Q_n(z)$ is disjoint from the support of the limiting zero measure of $P_n(z)$. 
\end{eg}
\begin{eg}
We will illustrate a more concrete example based on the above technique. Choose a polynomial $P$, whose zeros are in $\text{$\mathbb{D}$}_r$, where $r<1$. Define $Q_n(z)=P^n(z)-1$, then $Q_n'=nP^{n-1}(z)P'(z)$. If $z$ is a zero of $Q_n(z)$, then it satisfies $P^n(z)=1$, or $|P(z)|=1$. Therefore the limiting zero measure of $Q_n(z)$ is supported on the boundary of the lemniscate $\{z:|P(z)|\leq1\}$ of the polynomial $P$. The limiting zero measure for the sequence $\{Q_n\}_{n \geq 1}$ exists because $Q_n$ is the $nk$-th Chebyshev polynomial of the lemniscate of $P$. Hence the limiting zero measure is the equilibrium measure for the domain $\{z:|P(z)|\leq1\}$. For a detailed discussion on the relation between Chebyshev polynomials and equilibrium measures, the reader can refer Chapter 5 in \cite{ransford}. On the other side, if $z_1,z_2,\dots,z_k$ are the roots of the polynomial $P$, then the limiting zero distribution of $Q_n'$ will be $\frac{1}{k}\sum\limits_{i=1}^{k}\delta_{z_i}$. Hence the limiting measures of zeros and critical points of the given sequence of polynomials do not agree. 
\end{eg}

In this context we quote the question posed by Pemantle and Rivin in \cite{pemantle}. 

\begin{question}\label{Pemantle_question}
When are the zeros of $P_n'$ stochastically similar to the zeros of $P_n$?
\end{question}


\section{Critical points of random polynomials.}
To tackle the Question \ref{Pemantle_question}, $P_n$s can be considered to be random. The study of critical points of random polynomials through random zeros was initiated by Pemantle and Rivin in \cite{pemantle}. They considered a sequence of random polynomials whose zeros are i.i.d. with law $\mu$ having finite 1-energy and proved that the empirical law of critical points converge weakly to the same probability measure $\mu$. A similar result for  probability measures supporting on $S^1$ was proved by Subramanian \cite{sneha}. Kabluchko in \cite{kabluchko} proved the result without any assumption on $\mu$.

Before stating the above mentioned results we recall the modes of convergence for random measures.
\begin{definition}\label{modes of convergence}
 Let $\mbox{\textbf{M}}(\text{$\mathbb{C}$})$ be the set of probability measures on the complex plane, equipped with \textit{weak topology}. Let $\{\mu_n\}_{n\geq1}$ be a sequence in $\mbox{\textbf{M}}(\text{$\mathbb{C}$})$ and   $\mu \in \mbox{\textbf{M}}(\text{$\mathbb{C}$})$ we say,
\begin{itemize}
\item $\mu_n \xrightarrow{w} \mu$ in probability if $\lim\limits_{n\rightarrow\infty}\Pr(\mu_n \in N_\mu)=1$ for any  neighbourhood $N_\mu$ of $\mu$,

\item $\mu_n \xrightarrow{w} \mu$ almost surely if $\Pr(\lim\limits_{n\rightarrow\infty}\mu_n \in N_\mu)=1$ for any neighbourhood $N_\mu$ of $\mu$.
\end{itemize}
\end{definition}

 We now give precise statements of the results in  \cite{kabluchko} and \cite{pemantle}.

\begin{definition}
	Define the \textit{p-energy of $\mu$} to be
	
\[
	\mathcal{E}_\text{$p$}(\mu):=\left(\int\limits_{\mathbb{C}}\int\limits_{\mathbb{C}}\text{$\frac{1}{|z-w|^p}d\mu(z)d\mu(w)$}\right)^\textbf{$\frac{1}{p}$}
\]
\end{definition} 
\begin{theorem}[Pemantle-Rivin  ~\cite{pemantle}]\label{pemantle-rivin}Let $X_1,X_2,\dots $ be a sequence of i.i.d random variables from the probability measure $\mu$. Assume that $\mu$ has finite 1-energy. Let $P_n(z)=(z-X_1)(z-X_2)\dots(z-X_n)$, then the critical points measure $\text{$\mathscr{M}$}(P_n')\xrightarrow{w}\mu$   almost surely. 	
\end{theorem}
One limitation of the above result is that it is not applicable to probability measures that are supported on 1-dimensional subsets of the complex plane. But the result can be easily verified for probability measures supported on real line. By Rolle's theorem the critical points are interlaced between the roots of the polynomial. Hence the L\'{e}vy distance between the zeros measure $\text{$\mathscr{M}$}(P_n)$ and critical points measure $\text{$\mathscr{M}$}(P_n')$ is at most $\frac{1}{n}$. On the other side the zeros measure $\text{$\mathscr{M}$}(P_n)$ has a limiting measure $\mu$ which is the probability measure from which the random variables are drawn. Combining the previous two observations the result follows. Pemantle and Rivin in \cite{pemantle} conjectured that the statement of Theorem \ref{pemantle-rivin} is true without any assumptions on $\mu$. 
 
\begin{conjecture}[Pemantle-Rivin \cite{pemantle}]\label{pemantle_conjecture}
Let $X_1,X_2,\dots$ be i.i.d. random variables distributed according to a probability measure $\mu$ and $P_n(z):=(z-X_1)(z-X_2)\dots(z-X_n)$. Then $\text{$\mathscr{M}$}(P_n')\xrightarrow{w}\mu$  almost surely.
\end{conjecture}

 Kabluchko proved the conjecture of Pemantle and Rivin in a weak form.
\begin{theorem}[Kabluchko~\cite{kabluchko}]\label{kabluchko}
	 Let $X_1, X_2 ,\dots$ be i.i.d. random variables distributed according to a probability measure  $\mu$ and $P_n(z):=(z-X_1)(z-X_2)\dots(z-X_n)$. Then $\text{$\mathscr{M}$}(P_n')\xrightarrow{w}\mu$  in probability.
\end{theorem}

Further results concerning critical points and zeros of random polynomials are discussed below. In \cite{cheung} the authors (Pak-Leong Cheung, Tuen Wai Ng, Jonathan Tsai and SCP Yam) prove that the empirical law of zeros of the higher derivatives for the polynomial whose zeros are i.i.d. with law $\mu$ supported in $S^1$ converge to the same probability measure $\mu$. In \cite{cheung} the authors also obtain similar results for the zeros of generalized derivatives of polynomials. Similar results for critical points of characteristic polynomials of random matrix ensembles (Haar distributed on $O(n)$, $SO(n)$, $U(n)$, $Sp(n)$) are proved in \cite{orourke} by O'Rourke. 


%
%
 
%



In this section we will present two results concerning the zeros and critical points of the sequence of random polynomials. In the previous section we have seen examples of polynomials for which the limiting empirical distribution of zeros and critical points do not agree. Where as if  the zeros of the polynomial are chosen to be i.i.d. random variables, then the statement holds \cite{kabluchko}. These results bridge the gap between the two scenarios, i.e., we reduce the randomness in choosing the zeros and show that the statement holds. In Theorem \ref{thm1} we will start with two sequences of complex numbers which are asymptotically distributed according to a same probability measure. We also assume that the two sequences are sufficiently different (precise conditions are stated in the theorem). Then we construct a sequence of random numbers, whose terms are chosen independently at random from the corresponding terms of either of the sequences. If we make a sequence of polynomials whose zeros are the terms of the obtained random sequence, then  the limiting measure of the critical points of this sequence of polynomials will agree with that of the limiting measure of the sequences we started with.


We prove the result for a specific class of sequences which we call as \textit{log-Ces\'{a}ro-bounded} which is defined as follows.
\begin{definition}\label{def:log-cesaro}
We say a sequence of complex numbers $\{a_n\}_{n\geq1}$ to be \textit{log-Ces\'{a}ro-bounded} if the Ces\'{a}ro means of the positive part of their logarithms are bounded i.e., the sequence $\{\frac{1}{n}\sum_{k=1}^{n}\log_+|a_k|\}$ is bounded.
\end{definition}

\begin{eg}
Any bounded sequence is a log-Ces\'{a}ro bounded. 
\end{eg}

\begin{theorem}\label{thm1}
	Let $\{a_k\}_{k\geq1}$ and $\{b_k\}_{k\geq1}$ be two $\mu$-distributed and log-Ces\'{a}ro bounded sequences of complex numbers. Additionally assume that, $a_k \neq b_k$ for infinitely many $k$.
	Define the sequence of independent random variables $\xi_k$ such that $\xi_k = a_k $ or $b_k$ with equal probability, for $k\geq1$. Define the polynomials $P_n(z):=(z-\xi_1)(z-\xi_2)\dots(z-\xi_n)$. Then, $\text{$\mathscr{M}$}(P_n)\xrightarrow{w}\mu$ almost surely and $\text{$\mathscr{M}$}(P_n')\xrightarrow{w}\mu$ in probability.  
\end{theorem}

For the assertion of the above theorem to hold, it is necessary to assume that the two sequences differ in infinitely many terms. Suppose not, we may choose one of the sequence to be a sequence for which the assertion of the theorem doesn't hold. Since both the sequences differ only in finitely many terms, the resulting sequence will be same as that of the sequence for which the assertion doesn't hold, with non zero probability. Hence with positive probability the statement of the Theorem \ref{thm1} doesn't hold. The log-Ces\'{a}ro boundedness on the sequences is assumed to enable the proof. We don't have any strong reason for either of the cases whether it is necessary or not.

The Theorem \ref{thm1} can be used to obtain corollaries of the following form. Choose a deterministic sequence which is $\mu$-distributed and perturb each of its term by a random variable with diminishing variances. It can be obtained that the empirical measure of the critical points of the polynomial, made from the perturbed sequence also converge to the same limiting probability measure $\mu$.

\begin{corollary}\label{Symmetric perturbations}
Let $\{u_n\}_{n\geq1}$ be a $\mu$-distributed sequence and log-Ces\'{a}ro bounded sequence. Let $\{v_n\}_{n\geq1}$ be the sequence such that $v_n=u_n+\sigma_nX_n$, where $X_n$s are i.i.d  random variables satisfying $X_n\stackrel{d}{=}-X_n$, $\ee{\text{$|X_n|$}}<\infty$  and $\sigma_n \downarrow 0$, $\sigma_n\neq0$. Define the polynomial $P_n(z):=(z-v_1)(z-v_2)\dots(z-v_n)$. Then,  $\text{$\mathscr{M}$}(P_n)\xrightarrow{w}\mu$ almost surely and  $\text{$\mathscr{M}$}(P_n')\xrightarrow{w}\mu$  in probability.
\end{corollary}

\begin{remark}
In Corollary \ref{Symmetric perturbations}, we may choose the random variables $X_n$s to have complex Gaussian distribution or uniform distribution on unit disk centred at $0$. In the case of complex Gaussian distributed random variables we get the result for unbounded perturbations and in the case of uniformly distributed random variables the perturbations are bounded.
\end{remark}

It is an easy fact (Page 15 in \cite{manjubook}) that if $\{X_n\}_{n \geq 1}$ is a sequence of i.i.d random variables that are not identically $0$ such that $\ee{\textbf{$\log_+|X_1|$}}<\infty$, then $\limsup\limits_{n \rightarrow \infty}|X_n|^\frac{1}{n}=1$. 
A special case of Theorem \ref{kabluchko} can be obtained as a corollary of the Theorem \ref{thm1}. The special case being the one in which the probability measure $\mu$ in consideration has bounded $\log_+$-moment.

\begin{corollary}\label{corollary4:kabluchko}
	Let $\mu$ be any probability measure on $\mathbb{C}$  satisfying $\int\limits_{\mathbb{C}}\log_+|z|d\mu(z)<\infty$. Let $X_1,X_2,\dots, X_n$ be i.i.d random variables distributed according to $\mu$. Define the polynomials  $P_n(z):=(z-X_1)(z-X_2)\dots(z-X_n)$. Then,  $\text{$\mathscr{M}$}(P_n)\xrightarrow{w}\mu$ almost surely and  $\text{$\mathscr{M}$}(P_n')\xrightarrow{w}\mu$  in probability.
\end{corollary}

 Let $\{u_n\}_{n\geq1}$ be $\mu$-distributed sequence and $\{v_n\}_{n\geq1}$ be $\nu$-distributed sequence and both are log-Ces\'{a}ro bounded. We replace the terms in the first sequence with those of the second sequence each with probability $p>0$. Let the random sequence be $\{\xi_n\}_{n\geq1}$, define $P_n(z)=(z-\xi_1)(z-\xi_2)\dots(z-\xi_n)$. Then, the limiting empirical measures of zeros and critical points of the polynomial $P_n$ will agree. We state this as the following proposition.

\begin{proposition}\label{corollary5:thm1}
	Let $\{u_k\}_{k\geq1}$ be a $\mu$-distributed sequence of complex numbers and $\{v_k\}_{k\geq1}$ be a $\nu$-distributed sequence of complex numbers. Assume that both the sequences $\{u_k\}_{k\geq1}$ and $\{v_k\}_{k\geq1}$ are log-Ces\'{a}ro bounded and $u_k\neq v_k$ for infinitely many $k$. For $i\geq1$ define the sequence of independent random variables to be $\xi_i = u_i $ with probability $p$ and $v_i$ with probability $1-p$, where $0 < p<1$. Define the polynomials  $P_n(z):=(z-\xi_1)(z-\xi_2)\dots(z-\xi_n)$. Then,  $\text{$\mathscr{M}$}(P_n)\xrightarrow{w} p\mu+(1-p)\nu$ almost surely and  $\text{$\mathscr{M}$}(P_n')\xrightarrow{w} p\mu+(1-p)\nu$  in probability.
\end{proposition}

We have seen examples of sequence of polynomials for which the limiting measure of zeros and critical points do not agree. Consider the case of the sequence of polynomials whose $n$-th term is $P_n(z)=z^n-1$. We have seen that the limiting measure of zeros is the uniform probability measure on $S^1$ and that of critical points is dirac measure at $0$. The zeros of $P_n$ are the $n$-th roots of unity. The zeros are symmetrical and balanced in many respects. Removing any of these zeros can disturb this symmetry and be considered as a perturbation asymptotically. Define the sequence to be $\{Q_n\}_{n\geq1}$, where \[Q_n(z)=\frac{P_{n+1}(z)}{z-1}=z^n+z^{n-1}+\dots+1.\] 

It will be shown that the limiting zero measure of the sequence $\{Q_n\}_{n\geq1}$ is the uniform probability measure on $S^1$. The derivative of these polynomials is \[Q_n'(z)=nz^{n-1}+(n-2)z^{n-1}+\dots+1=\frac{nz^{n+1}-(n+1)z^n+1}{(z-1)^2}.\]
 We will show that the limiting zero measure of $(z-1)^2Q_n'(z)$ $(=nz^{n+1}-(n+1)z^n+1)$ is the uniform probability measure on $S^1$ which in turn gives the limiting zero measure for $Q_n'$. Fix any $r>1$, there is $N_r$ such that whenever $n>N_r$, for $|z|\geq r$ we have,
 \[|nz^{n+1}-(n+1)z^n+1|\geq|n|z|^{n+1}+1-(n+1)|z|^n|>0.\]
 Similarly fix any $r<1$, there is $N_r$ such that whenever $n>N_r$, for $|z|\leq r$ we have,
 \[|nz^{n+1}-(n+1)z^n+1|=|z|^{n+1}\bigg|n-\frac{n+1}{z}+\frac{1}{z^{n+1}}\bigg|\geq|z|^{n+1}\bigg|\big|\frac{n+1}{z}-n\big|-\frac{1}{|z|^{n+1}}\bigg|>0.\]
 
 Hence the limiting zero measure of the sequence $\{Q_n'\}_{n\geq1}$ is supported on $S^1$. If we show that asymptotically the angular distribution of the zeros of $Q_n'$ is uniform on $[0,2\pi)$, then it follows that the limiting zero measure of the sequence $\{Q_n'\}_{n\geq1}$ is the uniform probability measure on $S^1$. To show this we use a bound of Erd\"{o}s-Turan for the discrepancy between a probability measure and uniform measure on $S^1$. We will sate the inequality in the case where the two measures are counting probability measure zeros of polynomial and uniform probability measure on $S^1$.

\begin{theorem}[Erd\"{o}s-Turan~\cite{erdos-turan}]
Let $\{a_k\}_{0\leq k\leq N}$ be a sequence of complex numbers such that $a_0a_N\neq 0$ and let, \[P(z)=\sum\limits_{k=0}^{N}a_kz^k.\] 
Then, \[\bigg|\frac{1}{N}\nu_N(\theta,\phi)-\frac{\phi-\theta}{2\pi}\bigg|^2\leq\frac{C}{N}\log\bigg|\frac{\sum_{k=0}^{N}|a_k|}{\sqrt{|a_0a_N|}}\bigg|,\]
for some constant $C$ and $\nu_N(\theta,\phi):=\#\{z_k:\theta\leq\arg(z_k)<\phi\}$, where $z_1,z_2,\dots,z_N$ are zeros of $P(z)$.
\end{theorem}

Applying the above inequality for the polynomial $(z-1)^2Q_n'(z)$,  we get
\[\bigg|\frac{1}{n}\nu_n(\theta,\phi)-\frac{\phi-\theta}{2\pi}\bigg|^2\leq\frac{C}{n}\log\bigg|\frac{2n+2}{\sqrt{n}}\bigg|\stackrel{n\rightarrow \infty}{\longrightarrow}0.\]

Therefore the limiting zero measure of $Q_n'$ is uniform probability measure on $S^1$ which agrees with the limiting zero measure of $Q_n$. As an application of the forthcoming theorem we will see that if we choose random subsequence from a $\mu$-distributed sequence, then the limiting  distribution of zeros and critical points agree for the polynomials made from this random sequence.

The next result (Theorem \ref{thm2}) deals with counting the zeros and pole of a random rational function. The random rational function is defined as $L_n(z)=\sum\limits_{k=1}^{n}\frac{a_k}{z-z_k}$. In a special case where $\sum_{k=1}^{n}a_k=n$ and $a_k>0$ for every $k=1,2,\dots,n$, it is called generalized Sz.-Nagy derivative. For a classical derivative all $a_k$s are equal to $1$. It is mentioned in \cite{rahmanbook} that the motivation in studying generalized derivative is that many of the results for classical derivatives extend to the generalized derivatives.
\begin{theorem}\label{thm2}
	Let $a_1, a_2, \dots$ be i.i.d. random variables satisfying $\ee{\text{$|a_1|$}}\text{$< \infty$}$.  Let $\{z_n\}_{n\geq 1}$ be a sequence satisfying that for Lebesgue a.e. $z \in \mathbb{C}$ there exists a compact set $K_z$ with $d(z,K_z)>0$ such that there are infinitely many $z_k$'s in $K_z$, and there is a point $\omega$ that is not a limit point of $z_n$'s. Define $L_n(z):=  \frac{a_1}{z-z_1}+\frac{a_2	}{z-z_2}\dots+\frac{a_n}{z-z_n} $. Then $\frac{1}{n}\Delta \log(|L_n(z)|)\rightarrow 0$ in probability, in the sense of distributions.
\end{theorem}
 In the statement of the above theorem, there is a mention of the sequence $\{z_n\}_{n\geq1}$ satisfying that for Lebesgue a.e. $z \in \mathbb{C}$ there exists a compact set $K_z$ with $d(z,K_z)>0$ such that there are infinitely many $z_k$'s in $K_z$, and there is a point $\omega$ that is not a limit point of $z_n$'s. Several classes of sequences satisfy this condition. For example any bounded sequence or any sequence that is not dense and $\mu$-distributed for appropriate $\mu$ satisfies this condition. 

\begin{remark}
In Theorem \ref{thm2}, let $L_n(z)=\frac{Q_n(z)}{P_n(z)}$. Where $Q_n(z)$ id defined to be the generalized derivative of the polynomial $P_n$. Then Theorem \ref{thm2} asserts that $\frac{1}{n}\Delta\log|L_n(z)|\rightarrow0$, which in turn imply that $\text{$\mathscr{M}$}(Q_n)-\text{$\mathscr{M}$}(P_n)\rightarrow 0$ in the sense of distributions. If we assume that the sequence $\{z_k\}_{k\geq1}$ is $\mu$-distributed then it follows that the limiting measure of critical points converge to $\mu$.
\end{remark}

As an application of previous Theorem \ref{thm1}  we choose a $\mu$-distributed deterministic sequence and perturb it randomly and show that the empirical distribution of critical points is also $\mu$. Instead here we choose a random subsequence of a $\mu$-distributed sequence and show that the corresponding result holds. We state this result as the following corollary.
\begin{corollary}\label{corollary:thm2}
	Let $\{z_n\}_{n\geq1}$ be a $\mu$-distributed sequence that is not dense in $\mathbb{C}$, for a $\mu$ which is not supported on the whole complex plane. Choose a subsequence $\{z_{n_k}\}_{k \geq 1}$ at random that is, each of $z_n$ is part of subsequence with probability $p<1$ independent of others. Define the polynomials $P_k(z):=(z-z_{n_1})(z-z_{n_2})\dots(z-z_{n_k})$.  Then,  $\text{$\mathscr{M}$}(P_k)\xrightarrow{w}\mu$ almost surely and  $\text{$\mathscr{M}$}(P_k')\xrightarrow{w}\mu$  in probability.
\end{corollary}

\section{Proofs of corollaries and Proposition \ref{corollary5:thm1}.}
In Corollary \ref{Symmetric perturbations} we deal with perturbations of a $\mu$-distributed sequence. We expect that the perturbed sequence will also have the same limiting probability measure as of the original sequence. It is formally stated and proved in the following lemma.

\begin{lemma}\label{lemma:perturb_limit_measure}
Let $\{a_n\}_{n\geq1}$ be a $\mu$-distributed sequence, $\sigma_n\downarrow0$ and $X_1,X_2,\dots$ are i.i.d. random variables. Then, $\{a_n+\sigma_nX_n\}_{n\geq1}$ is a $\mu$-distributed sequence almost surely.
\end{lemma}
\begin{proof}
It is enough to show that for any $f\in C_c^\infty(\mathbb{C})$,
\[\frac{1}{n}\sum\limits_{k=1}^{n}\left(f(a_k)-f(a_k+\sigma_kX_k)\right)\rightarrow 0,\]
almost surely. Fix $\epsilon>0$,  choose $M$ such that $\Pr(|X_n|>M)<\epsilon.$ Then,
\begin{align}
\frac{1}{n}|\sum\limits_{k=0}^{n}(f(a_k)-f(a_k+\sigma_kX_k)| & \leq \frac{1}{n}\sum\limits_{k=1}^{n}|(f(a_k)-f(a_k+\sigma_kX_k))\text{$\mathbbm{1}$}\{|X_k|>M\}|\\&+\frac{1}{n}\sum\limits_{k=1}^{n}|(f(a_k)-f(a_k+\sigma_kX_k))\text{$\mathbbm{1}$}\{|X_k|\leq M\}|,\\
&\leq\frac{ 2||f||_\infty}{n} \sum\limits_{k=1}^{n}\text{$\mathbbm{1}$}\{|X_k|>M\}+\frac{1}{n}\sum\limits_{k=1}^{n}|\sigma_kX_k|||f'||_\infty,\\
&\leq \frac{ 2||f||_\infty}{n} \sum\limits_{k=1}^{n}\text{$\mathbbm{1}$}\{|X_k|>M\} + \frac{M||f'||_\infty}{n}\sum\limits_{k=1}^{n}\sigma_k. \label{eqn:perturb_limit_measure}
\end{align}
Using law of large numbers and $\sigma_n\downarrow0$in the above equation \ref{eqn:perturb_limit_measure} we have 
\[
\lim\limits_{n\rightarrow\infty}\frac{1}{n}\bigg|\sum\limits_{k=0}^{n}(f(a_k)-f(a_k+\sigma_kX_k)\bigg|  \leq 2||f||_\infty\epsilon.
\]
Because $\epsilon>0$ is arbitrary, we get  $\lim\limits_{n\rightarrow \infty}\frac{1}{n}\sum\limits_{k=1}^{n}\left(f(a_k)-f(a_k+\sigma_kX_k)\right)=0$.
\end{proof}

The main idea in proving the corollaries is that we condition the random sequences suitably, so that the resulting sequences satisfy the hypothesis of the Theorem \ref{thm1} and then apply to obtain the result. More formally, say we condition the sequence on the event $E$. Assume the conditioned sequence can be realized as a random sequence which satisfies the hypothesis of Theorem \ref{thm1}. Let $\nu_n^E$ be the empirical measure of the critical points of the degree-$n$ polynomial formed by conditioned sequence. Fix $\epsilon>0$, then

\begin{align}
\Pr\left(d(\nu_n,\mu)>\epsilon\right) &=\ee{\text{$\text{$\mathbbm{1}$}\{d(\nu_n,\mu)>\epsilon\}$}},\\
&=\ee{\cee{\text{$\text{$\mathbbm{1}$}\{d(\nu_n,\mu)>\epsilon\}$}}{\text{$E$}}},\\
&=\ee{\text{$\text{$\mathbbm{1}$}\{d(\nu_n^E,\mu)>\epsilon\}$}}. \label{eqn:convergence}
\end{align}

But from the assumption made above, for every $\epsilon>0$ we have, 
\[\ee{\text{$\text{$\mathbbm{1}$}\{d(\nu_n^E,\mu)>\epsilon\}$}}=\text{$\Pr\left(d(\nu_n^E,\nu)>\epsilon\right)\xrightarrow{n \rightarrow \infty}0$}.\]


Applying the dominated convergence theorem to \eqref{eqn:convergence} it follows that, for everey $\epsilon>0$
\[\Pr\left(d(\nu_n,\mu)>\epsilon\right)\xrightarrow{n \rightarrow \infty} 0.\]

Therefore it is justified that to show the convergence of probability measures, it is enough to show the convergence of conditioned probability measures almost surely.

To invoke the hypothesis of the Theorem \ref{thm1} we need to show that the perturbed sequence is also log-Ces\'{a}ro bounded. It will be proved in the following lemma.

We will use the following inequalities, whenever required.
\begin{align}
\log_+|ab| & \leq \log_+|a| + \log_+|b| \label{logplusprod}\\
\log_-|ab| & \leq \log_-|a| +\log_-|b| \label{logminusprod}\\
\log_+|a_1+a_2+ \dots + a_n| & \leq \log_+|a_1| + \log_+|a_2|+ \dots +\log_+|a_n| +\log(n)\label{logplussum}
\end{align}

\begin{remark}
The inequality \eqref{logplussum} is obtained by using the inequalities
$|a_1+\dots+a_n|\leq |a_1|+\dots+|a_n| \leq n\max\limits_{i\leq n}{|a_i|}$
 and 
$\log_+(\max\limits_{i\leq n}|a_i|)\leq \log_+|a_1|+\dots+\log_+|a_n|.$

\end{remark}

\begin{lemma}\label{lemma:log-cesaro-bounded}
	Let $\{a_n\}_{n\geq1}$ be a sequence that is log-Ces\'{a}ro bounded and $\{b_n\}_{n\geq1}$ be a sequence such that $b_n=a_n+\sigma_nX_n$, $\sigma_n\downarrow0$ and $X_1,X_2,\dots$ are i.i.d. random variables with $\ee{\text{$\log_+|X_1|$}}<\infty$. Then the sequence $\{b_n\}_{n\geq1}$ is also log-Ces\'{a}ro bounded.
\end{lemma}
\begin{proof}
	\begin{align}
	\frac{1}{n}\sum\limits_{k=1}^{n}\log_+|b_k| & \leq \frac{1}{n}\sum\limits_{k=1}^{n}(\log_+(|a_k|+|a_k-b_k|)),\\
	&\leq \frac{1}{n}\sum\limits_{k=1}^{n}\log_+|a_k|+\frac{1}{n}\sum\limits_{k=1}^{n}\log_+|\sigma_kX_k|+\log(2)\\
	&\leq \frac{1}{n}\sum\limits_{k=1}^{n}\log_+|a_k|+\frac{1}{n}\sum\limits_{k=1}^{n}\log_+|\sigma_k|+\frac{1}{n}\sum\limits_{k=1}^{n}\log_+|X_k|+\log(2) \label{eqn:lemma:log-cesaro:1}
	\end{align}
	The sequence $\{\frac{1}{n}\sum_{k=1}^{n}\log_+|\sigma_k|\}_{n\geq1}$ goes to $0$, because $\lim\limits_{n\rightarrow0}\sigma_n=0$. Using  law of large numbers and the fact that $\ee{\text{$\log_+|X_1|$}}<\infty$, the sequence $\{\frac{1}{n}\sum_{k=1}^{n}\log_+|X_k|\}_{n\geq1}$ is bounded almost surely. Combining \eqref{eqn:lemma:log-cesaro:1} and the above facts we get that the sequence $\frac{1}{n}\sum\limits_{k=1}^{n}\log_+|b_k|$ is bounded. This completes the proof.
\end{proof}

\begin{proof}[Proof of Corollary \ref{Symmetric perturbations}]
	
	Fix $r_n$ and $\theta_n$ for $n \geq 1$. Choose  $E=\{w:X_n(w)=\pm r_ne^{i\theta_n} \linebreak\text{ for }n\geq1\}$. Because $X_n$s are symmetric random variables,  the $n^{th}$ term of the resulting sequence will be $u_n + \sigma_nr_ne^{i\theta_n}$ or $u_n - \sigma_nr_ne^{i\theta_n}$ with equal probability independent of other terms. Choose $a_n=u_n+\sigma_nr_ne^{i\theta_n}$ and $b_n=u_n-\sigma_nr_ne^{i\theta_n}$.	 We need to show that almost surely the sequences $\{a_n\}_{n\geq1}$ and $\{b_n\}_{n\geq1}$ satisfy the hypotheses of the Theorem \ref{thm1}. It follows from Lemmas \ref{lemma:perturb_limit_measure} and \ref{lemma:log-cesaro-bounded} the sequences $\{a_n\}_{n \geq 1}$ and $\{b_n\}_{n \geq 1}$ are $\mu$-distributed and log-Ces\'{a}ro bounded almost surely. 
\end{proof}

\begin{proof}[Proof of Corollary \ref{corollary4:kabluchko}]
	If $\mu$ is a degenerate probability measure then the result is trivial to verify. If $\mu$ is not deterministic then choose two independent sequences of random numbers $\{a_n\}_{n\geq1}$ and $\{b_n\}_{n\geq1}$, where $a_n$s and $b_n$s are i.i.d random numbers obtained from measure $\mu$. Choose $X_n= a_n$ or $b_n$ with equal probability independent of other terms, then $\{X_n\}_{n\geq1}$ is a sequence of i.i.d random variables distributed according to probability measure $\mu$. Using the hypothesis $\int\limits_{\mathbb{C}}\log_+|z|d\mu(z)<\infty$ and applying law of large numbers for the random variables $\{\log_+|X_n|\}_{n \geq 1}$, we get that the sequences $\{a_n\}_{n\geq1}$ and $\{b_n\}_{n\geq1}$ are log-Ces\'{a}ro bounded almost surely. Therefore the constructed sequences satisfy the hypothesis of the Theorem \ref{thm1}. 
\end{proof}
Before proving Proposition \ref{corollary5:thm1} we will prove the following lemma which give the limiting empirical measure of random sequence whose terms are drawn from either of two deterministic sequences.
\begin{lemma}\label{random_seq_limit}
Let $\{a_k\}_{k\geq1}$ and $\{b_k\}_{k\geq1}$ be two sequences which are $\mu$ and $\nu$ distributed respectively. Define a random sequence $\{\xi_k\}_{k\geq1}$, where $\xi_k=a_k$ with probability $p$ and $\xi_k=b_k$ with probability $1-p$. Then $\mu_n=\frac{1}{n}\sum\limits_{k=1}^{n}\delta_{\xi_k}$ weakly converge to $\lambda=p\mu+(1-p)\nu$ almost surely.
\end{lemma}
\begin{proof}
It is enough to show that for any open set $U\subset\mathbb{C}$, $ \frac{1}{n}\sum\limits_{k=1}^{n}\text{$\mathbbm{1}$}\left\{\xi_k \in U\right\}$ converge to $\lambda(U) $ almost surely. But from a version of law of large numbers we know that if $X_1,X_2,\dots$ are independent random variables (not necessarily identical), then 
\[
\frac{1}{n}\sum\limits_{k=1}^{n}\left(X_k-\ee{\text{$X_k$}}\right)\xrightarrow{a.s}0
\]
provided that $\sum\limits_{k=1}^{\infty}\frac{1}{k^2}\var{X\textbf{$_k$}}<\infty$. Applying this to the random variables $\textbf{$\mathbbm{1}$}\left\{\xi_k \in U\right\}$ we get that $\frac{1}{n}\sum\limits_{k=1}^{n}\text{$\mathbbm{1}$}\left\{\xi_k \in U\right\}$ converge to $\lambda(U)$ almost surely.
\end{proof}
\begin{proof}[Proof of Proposition \ref{corollary5:thm1}]
For $i\geq1$ choose a sequence of independent random variables to be $a_k$ which assumes values $u_k $ with probability $p$ and $v_k$ with probability $1-p$. Independent of this sequence choose another sequence of independent random variables whose terms are $b_k = u_k $ with probability $p$ and $v_k$ with probability $1-p$. The terms of the sequences $\{a_n\}_{n\geq1}$ and $\{b_n\}_{n\geq1}$ satisfy, 
\begin{align*}
\log_+|a_n|\leq \log_+|u_n|+\log_+|v_n|,\\
\log_+|b_n|\leq \log_+|u_n|+\log_+|v_n|.
\end{align*}
Because the sequences $\{u_n\}_{n\geq1}$ and $\{v_n\}_{n\geq1}$ are log-Ces\'{a}ro bounded, it follows from above inequalities, the sequences $\{a_n\}_{n\geq1}$ and $\{b_n\}_{n\geq1}$  are also log-Ces\'{a}ro bounded. 


Therefore, from the above arguments and Lemma \ref{random_seq_limit} the sequences $\{a_n\}_{n\geq1}$ and $\{b_n\}_{n\geq1}$ satisfy the hypothesis of the Theorem \ref{thm1} almost surely. Hence the corollary is proved.
\end{proof}
\begin{proof}[Proof of Corollary \ref{corollary:thm2}]

	Choose $a_1,a_2,\dots$ be i.i.d $\mbox{Bernoulli}(p)$ random variables. Let $\{k_n\}_{n\geq1}$  be a random sequence such that $a_{k_n}=1$ and $a_\ell=0$ whenever $\ell \notin \{k_1,k_2,\dots\}$. Define $L_n^{(1)}(z)=L_{k_n}(z)=\frac{P_n'(z)}{P_n(z)}$. It is enough to show that $\frac{1}{n}\Delta\log|L_{k_n}(z)|\rightarrow 0$ in probability. The sequences $\{a_n\}_{n\geq1}$ and $\{z_n\}_{n\geq1}$ satisfy the hypothesis of the Theorem \ref{thm2}. Therefore $\frac{1}{n}\Delta \log|L_n(z)|\rightarrow 0$ in probability. Because $\{L_{k_n}^{(1)}(z)\}_{n\geq1}$ is a subsequence of $\{L_n(z)\}_{n\geq1}$ it follows that $\frac{1}{k_n}\Delta \log|L_{k_n}^{(1)}(z)|\rightarrow 0$ in probability. Because $k_n$ is a negative binomial random variable with parameters $(n,p)$, we have $\frac{k_n}{n}\rightarrow p$ almost surely. Therefore, $\frac{1}{n}\Delta \log|L_{k_n}^{(1)}(z)|\rightarrow 0$ in probability.
\end{proof}
In the next chapter we provide proofs for both the Theorems \ref{thm1} and \ref{thm2}.


\chapter{Proofs of Theorems \ref{thm1} and  \ref{thm2}.}
\label{ch:proofs5}
\section{Outline of proofs.}
The proofs here are adapted from the proof of  Kabluchko's theorem as presented in \cite{kabluchko}. The proofs involve in analysing the function $L_n(z)$. In case of Theorem \ref{thm1} define $L_n(z)=\frac{P_n'(z)}{P_n(z)} = \sum\limits_{k=1}^{n}\frac{1}{z-\xi_k}$.
We shall prove the theorems by showing that the hypotheses of the Theorems  \ref{thm1} and \ref{thm2} imply the following three statements. 

\begin{align}
&\text{For Lebesgue a.e. $z \in \mathbb{C}$ }  \text{  and for every }\epsilon>0, \lim\limits_{n\rightarrow \infty}\Pr\left(\frac{1}{n}\log|L_n(z)|>\epsilon\right)=0.  \tag{A1} \label{A1}\\
&\text{For Lebesgue a.e. $z \in \mathbb{C}$ } \text{ and for every }\epsilon>0, \lim\limits_{n\rightarrow \infty}\Pr\left(\frac{1}{n}\log|L_n(z)|<-\epsilon\right)=0. \tag{A2} \label{A2}\\
&\text{For any }r>0, \text{the sequence }\left\{\int_{\text{$\mathbb{D}$}_r}\frac{1}{n^2}\log^2|L_n(z)|\right\}_{n\geq 1} \text{ is tight.} \tag{A3} \label{A3}
\end{align}

Statements \eqref{A1} and \eqref{A2} assert that $\frac{1}{n}\log|L_n(z)|$ converge to $0$ in probability. Statement \eqref{A3} assert that the sequence $\{\int\limits_{\text{$\mathbb{D}$}_r}\frac{1}{n^2}\log^2|L_n(z)|\}_{n \geq 1}$ is tight. A lemma of Tao and Vu links the above two facts to yield that $\{\int\limits_{\text{$\mathbb{D}$}_r}\frac{1}{n}\log|L_n(z)|\}_{n \geq 1}$ converge to $0$ in probability. We state this lemma below.

\begin{lemma}[Lemma~3.1 in~\cite{taovu}]\label{lem:tao_vu}
	Let $(X,\mathcal{A},\nu)$ be a finite measure space and $f_n:X\to \mathbb{R}$, $n\geq 1$  random functions which are defined over a probability space $(\Omega, \mathcal{B}, \mathbb{P})$ and are jointly measurable with respect to $\mathcal{A}\otimes \mathcal{B}$.
	Assume that:
	\begin{enumerate}
		\item For $\nu$-a.e.\ $x\in X$ we have $f_n(x)\to 0$ in probability, as $n\to\infty$.
		\item For some $\delta>0$, the sequence $\int_X |f_n(x)|^{1+\delta} d\nu(x)$ is tight.
	\end{enumerate}
	Then, $\int_X f_n(x)d\nu(x)$ converge in probability to $0$.
\end{lemma}
 Thus it follows from the above assertions \eqref{A1}, \eqref{A2}, \eqref{A3} and Lemma \ref{lem:tao_vu}, that  $\int\limits_{\text{$\mathbb{D}$}_r}\frac{1}{n}\log|L_n(z)|dm(z) \rightarrow 0$ in probability for any $r>0$. Choose any $f\in C_c^{\infty}(\mathbb{C})$, assume that $\mbox{support}(f) \subseteq \text{$\mathbb{D}$}_r$ and define $f_n(z)=\frac{1}{n}\left(\log|L_n(z)|\right)\Delta f(z)$. Because $f$ is a bounded function and $\frac{1}{n}\log|L_n(z)|$ satisfy the hypothesis of Lemma \ref{lem:tao_vu}, the functions $f_n$ also satisfy the hypothesis of Lemma \ref{lem:tao_vu}. Therefore we get that $\int\limits_{\text{$\mathbb{D}$}_r}f_n(z)dm(z)\rightarrow 0$ in probability.  Applying Green's theorem twice we have the identity,
\[\int_{\text{$\mathbb{D}$}_r}^{}f(z)\Delta\frac{1}{n}\log|L_n(z)| = \int_{\text{$\mathbb{D}$}_r}^{}\frac{1}{n}\log|L_n(z)|\Delta f(z)dm(z).
\] 
The left hand side of the above integral is defined in the sense of distributions. Therefore it follows that
$\int_{\text{$\mathbb{D}$}_r}^{} f(z)\frac{1}{n}\Delta\log|L_n(z)|\rightarrow 0 $ in probability. This suffices for Theorem \ref{thm2}. We complete the proof of Theorem \ref{thm1} by the following arguments. In the sense of distributions we have

\begin{equation}\label{eqn:distribution}
\int_{\text{$\mathbb{D}$}_r}^{} f(z)\frac{1}{n}\Delta\log|L_n(z)| = \frac{1}{n}\sum\limits_{k=1}^{n}f(\xi_k)-\frac{1}{n}\sum\limits_{k=1}^{n-1}f(\eta_k^{(n)})
\end{equation}

From Lemma \ref{random_seq_limit} it follows that the sequence $\{\xi_n\}_{n\geq1}$ is $\mu$-distributed. Hence \linebreak $\frac{1}{n}\sum_{k=1}^{n}f(\xi_k) \rightarrow \int\limits_{\text{$\mathbb{D}$}_r}f(z)d\mu(z)$ almost surely. Therefore from \eqref{eqn:distribution} we get, \begin{equation}
\frac{1}{n}\sum_{k=1}^{n-1}f(\eta_k^{(n)}) \rightarrow \int\limits_{\text{$\mathbb{D}$}_r}f(z)d\mu(z) \text{  in probability.} \label{proof_1}
\end{equation}  Because for any $f \in C_c^\infty(\text{$\mathbb{C}$})$ and $\epsilon>0$, the sets of the form $\{\mu:|\int\limits_{\text{$\mathbb{C}$}}f(z)d\mu(z)|<\epsilon\}$ form an open base at origin, from Definition \ref{modes of convergence} and \eqref{proof_1} it follows that  $\frac{1}{n-1}\sum_{k=1}^{n-1}\delta_{\eta_i^{(n)}} \xrightarrow{w} \mu$ in probability.

We show \eqref{A1}, by obtaining moment bounds for $L_n(z)$. To show \eqref{A2} we will use a concentration bound for the function $L_n(z)$. In either of the Theorems \ref{thm1} and \ref{thm2}, observe that $L_n(z)$ is a sum of independent random variables. Kolmogorov-Rogozin inequality gives the concentration bounds for sums of independent random variables.  A version of Kolmogorov-Rogozin inequality which will be used later in the proofs is stated below.

\paragraph{Kolmogorov-Rogozin inequality (multi-dimensional version)}\label{KR-Inequality} [Corollary 1. of Theorem 6.1 in \cite{KR1}.]
Let $X_1,X_2, \dots $ be independent random vectors in $\mathbb{R}^\text{$n$}$. Define the concentration function, 
\begin{equation}
Q(X,\delta) := \sup_{a\in \mathbb{R}^\text{$n$}}\Pr(X \in B(a,\delta)).
\end{equation}
Let $\delta_i \leq \delta$ for each $i$, then 
\begin{equation}
Q(X_1+\dots + X_n,\delta) \leq \frac{C\delta}{\sqrt{\sum_{i=1}^{n}\delta_i^2(1-Q(X_i,\delta_i))}}.\label{kol-rog-ineq}
\end{equation}

 It remains to show that the hypotheses of Theorems \ref{thm1} and \ref{thm2} imply \eqref{A1}, \eqref{A2} and \eqref{A3}. We show this in the subsequent sections.
 
\section{Proof of Theorem \ref{thm1}}

 In the following lemma we show that the hypothesis of the Theorem \ref{thm1} imply \eqref{A1}.
\begin{lemma}\label{momentbound}
	Let $L_n(z)=\sum\limits_{k=1}^{n}\frac{1}{z-\xi_k}$ where $\xi_ks$ are as in the Theorem \ref{thm1}. Then for any $\epsilon > 0$, 
	and for Lebesgue a.e. $z \in \mathbb{C}$ we have $\lim\limits_{n\rightarrow \infty}\Pr(\frac{1}{n}\log|L_n(z)|\geq \epsilon)= 0$.
\end{lemma}
\begin{proof}
	Define  $A_n^{\epsilon}=\bigcup\limits_{k=1}^{n}\{z:|z-a_k|<e^{-n\epsilon} \text{ or } |z-b_k|<e^{-n\epsilon}\}$ and $F^{\epsilon}=\limsup\limits_{n \rightarrow \infty} A_n^{\epsilon}$, then $F^\epsilon$ are decreasing sets in $\epsilon$. For these sets we have $\sum\limits_{n=1}^{\infty}m(A_n^{\epsilon}) \leq \sum\limits_{n=1}^{\infty}2\pi ne^{-2n\epsilon} < \infty$, where $m$ is Lebesgue measure on complex plane. Applying Borel-Cantelli lemma to the sequence $\{A_n^{\epsilon}\}_{n\geq 1}$ we get $m(F^{\epsilon})=0$. Because $F^\epsilon$ are decreasing sets in $\epsilon$, we have that if  $F=\bigcup\limits_{\epsilon>0}F^{\epsilon}$, then $m(F)=0$. Choose $z \in F^c$, there is $N_z^{\epsilon}$ such that for any $n>N_z^{\epsilon}$ we have $z \notin A_n^{\epsilon}$. Therefore $\frac{1}{|z-\xi_n|}>e^{n\epsilon}$ is satisfied only for finitely many $n$. Hence we have $|L_n(z)|<M+ne^{n\epsilon}$, where $M$ is the finite random number obtained from the terms for which the inequality $\frac{1}{|z-\xi_n|}>e^{n\epsilon}$ is violated. It follows from here  $\limsup\limits_{n\rightarrow \infty}\frac{1}{n}\log|L_n(z)|<\epsilon$ almost surely. Therefore for $z\notin \mathbb{C}$, we have $\lim\limits_{n\rightarrow \infty}\Pr(\frac{1}{n}\log|L_n(z)|\geq \epsilon)= 0$.
\end{proof}
\begin{remark}
In proof of Lemma \ref{momentbound}, we have proved a stronger statement that for Lebesgue almost every $z$, $\limsup\limits_{n\rightarrow \infty}\frac{1}{n}\log|L_n(z)|=0$ almost surely.
\end{remark}
We will use the null set $F$ defined in the proof of above Lemma \ref{momentbound} in the proofs of subsequent lemmas. In the forthcoming lemma we establish \eqref{A2}.

\begin{lemma}\label{kolmogorov-rogozin}
	Let $L_n(z)=\sum\limits_{k=1}^{n}\frac{1}{z-\xi_k}$ where $\xi_ks$ are as in the Theorem \ref{thm1}. Then for any $\epsilon > 0$, 
	and almost every $z$ we have $\lim\limits_{n\rightarrow \infty}\Pr(\frac{1}{n}\log|L_n(z)|\leq -\epsilon)= 0$.
\end{lemma}
\begin{proof}
	Fix $z \in F^c$, where $F$ is as defined in proof of lemma \ref{momentbound}.
	From  Kolmogorov-Rogozin inequality \eqref{kol-rog-ineq} and taking $\delta_i=\delta=e^{-n\epsilon}$ we have, 
	
	\begin{equation}\label{kreq}
	\Pr\left(\bigg|\sum\limits_{k=1}^{n}\frac{1}{z-\xi_k}\bigg|<e^{-n\epsilon}\right) \leq \frac{C}{\sqrt{\sum_{k=1}^{n}(1-Q(\frac{1}{z-\xi_k},e^{-n\epsilon}))}}.
	\end{equation}
	We shall show that $\sum_{k=1}^{n}(1-Q(\frac{1}{z-\xi_k},e^{-n\epsilon}))$ goes to $\infty$. Observe that,
	\begin{align}
	Q\left(\frac{1}{z-\xi_k},e^{-n\epsilon}\right) =& \sup\limits_{\alpha \in \mathbb{C}}\Pr\left( \bigg|\frac{1}{z-\xi_k}-\alpha\bigg|<e^{-n\epsilon} \right)
	\leq \frac{1}{2},   
	\end{align}
	whenever $|\frac{1}{z-a_k}-\frac{1}{z-b_k}| > 2e^{-n\epsilon}$. 
%

Define $S_n=\{k\leq n:|\frac{1}{z-a_k}-\frac{1}{z-b_k}| > 2e^{-n\epsilon}\}$. Notice that if $a_k \neq b_k$, then there is $N_k$ such that whenever $n>N_k$, we have $k\in S_n$. Because $a_k\neq b_k$ for infinitely many $k$, $|S_n|$ increases to infinity as $n \rightarrow \infty$. The denominator on the right hand side of \eqref{kreq} is at least $\sqrt{\frac{|S_n|}{2}}$.
Therefore $\Pr\left(\bigg|\sum\limits_{k=1}^{n}\frac{1}{z-\xi_k}\bigg|<e^{-n\epsilon}\right)\leq \frac{C\sqrt{2}}{\sqrt{|S_n|}}\rightarrow 0$, as $n \rightarrow \infty$. This completes the proof of the lemma.
\end{proof}
It remains to prove the tightness for the sequence $\{\int_{\text{$\mathbb{D}$}_r}\frac{1}{n^2}\log^2|L_n(z)|dm(z)\}_{n\geq1}$ and will be proved in the following lemma.
\begin{lemma}\label{tight}
	Let $L_n(z):=\sum\limits_{k=1}^{n}\frac{1}{z-\xi_k}$, where $\xi_ks$ are as in the Theorem \ref{thm1}. Then, for any $r>0$, the sequence $\{\int_{\text{$\mathbb{D}$}_r}\frac{1}{n^2}\log^2|L_n(z)|dm(z)\}_{n\geq1}$ is tight. 
\end{lemma}
\begin{proof}

	We will first decompose $\log|L_n(z)|$ into its positive and negative parts and analyze them separately. Let $\log|L_n(z)|=\log_+|L_n(z)|-\log_-|L_n(z)|$. Then,
	\[\int_{\text{$\mathbb{D}$}_r}\frac{1}{n^2}\log^2|L_n(z)|dm(z)= \int_{\text{$\mathbb{D}$}_r}\frac{1}{n^2}\log_+^2|L_n(z)|dm(z)+\int_{\text{$\mathbb{D}$}_r}\frac{1}{n^2}\log^2_-|L_n(z)|dm(z).\]
Using \eqref{logplussum}, we get,
	\begin{align}
	\int_{\text{$\mathbb{D}$}_r}\frac{1}{n^2}\log_+^2|L_n(z)|dm(z) & = \int_{\text{$\mathbb{D}$}_r}\frac{1}{n^2}\log_+^2\bigg|\sum\limits_{k=1}^{n}\frac{1}{z-\xi_k}\bigg|dm(z),\\
	& \leq \int_{\text{$\mathbb{D}$}_r}\frac{1}{n^2}\left(\sum\limits_{k=1}^{n}\log_+\bigg|\frac{1}{z-\xi_k}\bigg|+\log(n)\right)^2dm(z).
	\end{align}
	
	Using the Cauchy-Schwarz inequality $(a_1+a_2+\dots+a_n)^2\leq n(a_1^2+a_2^2+\dots+a_n^2)$ for the above, we get,
	
	\begin{align}
	\int_{r\text{$\mathbb{D}$}}\frac{1}{n^2}\log_+^2|L_n(z)|dm(z) 
	& \leq \int_{\text{$\mathbb{D}$}_r}\frac{n+1}{n^2}\left(\sum\limits_{k=1}^{n}\log_+^2\bigg|\frac{1}{z-\xi_k}\bigg|+\log^2(n)\right)dm(z),\\
	& = \frac{n+1}{n^2}\sum\limits_{k=1}^{n}\int_{\text{$\mathbb{D}$}_r}\log_-^2|z-\xi_k|dm(z) + \frac{n+1}{n^2}\log^2(n)\pi r^2.\label{eqn:lemma:tight:1}
	\end{align}
	 Because Lebesgue measure on complex plane is translation invariant, we have $$\int_{\text{$\mathbb{D}$}_r}\log_-^2|z-\xi|dm(z)=\int_{\xi+\text{$\mathbb{D}$}_r}\log_-^2|z|dm(z)\leq\int_{\text{$\mathbb{D}$}_1}\log_-^2|z|dm(z)<\infty.$$ Therefore $\sup\limits_{\xi\in \text{$\mathbb{C}$}}\int_{K}\log^2|z-\xi|dm(z)<\infty$ for any compact set $K \subset \mathbb{C}$ it can be seen that each of the terms in the final expression \eqref{eqn:lemma:tight:1} are bounded. Hence the sequence $\{\int_{\text{$\mathbb{D}$}_r}\frac{1}{n^2}\log_+^2|L_n(z)|dm(z)\}_{n\geq1}$ is bounded.
	
	We will now show that the sequence $\{\int_{\text{$\mathbb{D}$}_r}\frac{1}{n^2}\log_-^2|L_n(z)|dm(z)\}_{n\geq1}$ is bounded.
	Let $P_n(z)=\prod\limits_{k=1}^{n}(z-\xi_k)$ and $P_n'(z)=n\prod\limits_{k=1}^{n-1}(z-\eta_{k}^{(n)})$. Applying inequality \eqref{logminusprod} and Cauchy-Schwarz inequality we get,
	
	\begin{align}
	\int_{\text{$\mathbb{D}$}_r}\frac{1}{n^2}\log_-^2|L_n(z)|dm(z) & = \int_{\text{$\mathbb{D}$}_r}\frac{1}{n^2}\log_-^2\bigg|\frac{P_n'(z)}{P_n(z)}\bigg|dm(z),\\
	& \leq  \int_{\text{$\mathbb{D}$}_r}\frac{2}{n^2}\log_-^2|P_n'(z)|dm(z)+ \int_{\text{$\mathbb{D}$}_r}\frac{2}{n^2}\log_-^2\bigg|\frac{1}{P_n(z)}\bigg|dm(z).
	\end{align}
	Again applying inequalities \eqref{logminusprod}, \eqref{logplusprod}, \eqref{logplussum} and Cauchy-Schwarz inequality to the above we obtain,
	\begin{align}
		\int_{\text{$\mathbb{D}$}_r}\frac{1}{n^2}\log_-^2 & |L_n(z)|dm(z)\\ 	
	& \leq \int_{\text{$\mathbb{D}$}_r}\frac{2}{n^2}\left(\sum\limits_{k=1}^{n-1}\log_-|z-\eta_{k}^{(n)}|\right)^2dm(z) + \int_{\text{$\mathbb{D}$}_r}\frac{2}{n^2}\left(\sum_{k=1}^{n}\log_+|z-\xi_k|\right)^2dm(z),\\
	& \leq \frac{2}{n}\sum\limits_{k=1}^{n-1}\int_{\text{$\mathbb{D}$}_r}\log_-^2|z-\eta_{k}^{(n)}|dm(z) \label{eqn:lemma:tight1}\\&+ 2\int_{\text{$\mathbb{D}$}_r}\left(\log(2)+\log_+|z|+\frac{1}{n}\sum\limits_{k=1}^{n}\log_+|\xi_k|\right)^2dm(z).
	\label{eqn:lemma:tight2}
	\end{align}
	
	From the hypothesis, we have that both the sequences $\{a_n\}_{n\geq1}$ and $\{b_n\}_{n\geq 1}$ are log-Ces\'{a}ro bounded, which in turn implies that $\{\xi_n\}_{n\geq1}$ is also log-Ces\'{a}ro bounded, almost surely. Hence the integrand in \eqref{eqn:lemma:tight2} is bounded uniformly in $n$. Therefore \eqref{eqn:lemma:tight2} is bounded uniformly in $n$. Using the fact that $\sup\limits_{\xi\in \text{$\mathbb{C}$}}\int_{K}\log^2|z-\xi|dm(z)<\infty$ , we get \eqref{eqn:lemma:tight1} is bounded uniformly in $n$.
	 From the above facts we get that the sequence $\left\{\frac{1}{n^2}\int_{\text{$\mathbb{D}$}_r}\log^2|L_n(z)|\right\}_{n\geq1}$ is tight. 
\end{proof}
Lemmas \ref{momentbound}, \ref{kolmogorov-rogozin}, \ref{tight} show that the statements \eqref{A1}, \eqref{A2} and \eqref{A3} are satisfied. Hence the Theorem \ref{thm1} is proved.

\section{Proof of Theorem \ref{thm2}}

We will prove the theorem when $\omega=0$ i.e, $0$ is not a limit point of the sequence $\{z_n\}_{n\geq1}$. For other cases we can translate all the points by $\omega$ and apply the theorem. We will first prove a general lemma for sequences of numbers which will later be used in proving the subsequent lemmas.
\begin{lemma}\label{liminf}
	Given any sequence $\{z_k\}_{k\geq1}$, where $z_k \in \mathbb{C}$, $\liminf\limits_{n \rightarrow \infty}\left(\inf\limits_{|z|=r}|z-z_n|^{\frac{1}{n}}\right) \geq 1$ for Lebesgue a.e. $r \in \mathbb{R^+}$ w.r.t Lebesgue measure. 
\end{lemma}
\begin{proof}
	Fix $\epsilon > 0$ and let $A_n = \{r>0:\inf\limits_{|z|=r}|z-z_n|< (1-\epsilon)^{n} \}$. Let $m$ denote the Lebesgue measure on the complex plane. Then,
	\begin{align}
	m\left(\left\{r>0 :\liminf\limits_{n \rightarrow \infty}(\inf\limits_{|z|=r}|z-z_n|^\frac{1}{n}) \leq (1-\epsilon)\right\}\right) & =  m\left(\limsup\limits_{n \rightarrow \infty}A_n\right) \\
	& \leq  \lim\limits_{k \rightarrow \infty}m\left(\mathop{\cup}_{n \geq k} A_n\right)  
	\end{align}
	If $r \in A_k, $ then from the definition of $A_k$ we have that $r \in [|z_k|-(1-\epsilon)^k,|z_k|+(1-\epsilon)^k] $. Hence we get,	
	\begin{align}
	m&\left(\left\{r>0 :\liminf\limits_{n \rightarrow \infty}(\inf\limits_{|z|=r}|z-z_n|^\frac{1}{n}) \leq (1-\epsilon)\right\}\right)\\ & \leq \lim\limits_{k\rightarrow \infty } \sum_{n=k}^{\infty}m\left(\left\{r:|z_n|-(1-\epsilon)^{n} \leq r \leq |z_n|+(1-\epsilon)^n\right\}\right)\\
	& \leq  \lim\limits_{k \rightarrow \infty } \sum_{n=k}^{\infty} 2(1-\epsilon)^n = 0
	\end{align}
	The above is true for every $\epsilon >0$, therefore $\liminf\limits_{n \rightarrow \infty}\left(\inf\limits_{|z|=r}|z-z_n|^{\frac{1}{n}}\right) \geq 1$ outside an exceptional set $E\subset\mathbb{R}^\text{$+$}$ whose Lebesgue measure is $0$.  
\end{proof}
Define the set $F=\{z:\liminf\limits_{n \rightarrow \infty}|z-z_n|^{\frac{1}{n}} < 1\}$. Because $0$ is not a limit point of $\{z_n\}_{n\geq1}$, we have $\liminf\limits_{n\rightarrow\infty}|z_n|^\frac{1}{n}\geq1$. Hence $0\notin F$. For $|z|=r$, we have \[\liminf\limits_{n \rightarrow \infty}|z-z_n|^\frac{1}{n}\geq\liminf\limits_{n \rightarrow \infty}\left(\inf\limits_{|z|=r}|z-z_n|^{\frac{1}{n}}\right).\]
Hence $F\subseteq \{z:|z|=r, r\in E\}$  and by invoking Fubini's theorem we get $m( \{z:|z|=r, r\in E\})=0$. From the above two observations it follows that $m(F)=0$.

The following lemma shows that the hypothesis of the Theorem \ref{thm2} implies \eqref{A1}.
\begin{lemma}\label{momentthm2}
	Let $L_n(z)$ be as in the Theorem \ref{thm2}. Then for any $\epsilon>0$, and Lebesgue a.e. $z \in \mathbb{C}$,
	$$\limsup\limits_{n \rightarrow \infty}\frac{1}{n}\log|L_n(z)|< \epsilon$$ almost surely.
\end{lemma}
\begin{proof}
	From the hypothesis, Lemma \ref{liminf} and Using Markov's inequality we get  $$\sum_{n=1}^{\infty}\Pr\left(\sup\limits_{|z|=r}\big|\frac{a_n}{z-z_n}\big|>e^{n\epsilon}\right) \leq \sum_{n=1}^{\infty}\sup\limits_{|z|=r}\frac{\ee{\text{$|a_n|$}}}{|z-z_n|e^{n\epsilon}}.$$ Denoting $t_n(r)=\sup\limits_{|z|=r}\bigg|\frac{1}{z-z_n}\bigg|$ we have
	\begin{equation}
\sum_{n=1}^{\infty}\sup\limits_{|z|=r}\frac{\ee{\text{$|a_n|$}}}{|z-z_n|e^{n\epsilon}}=\sum_{n=1}^{\infty}\frac{\ee{\text{$|a_n|$}}}{e^{n\epsilon}}t_n(r). \label{eqn:power_series}
	\end{equation} 
	  Because $a_n$s are i.i.d. random variables, $\ee{\text{$|a_n|$}}=\ee{\text{$|a_1|$}}$. Using the root test for the convergence of sequences and the Lemma \ref{liminf}, it follows that the right hand side of \eqref{eqn:power_series} is convergent for Lebesgue a.e. $r \in (0,\infty)$. Invoking Borel-Cantelli lemma we can say that $\sup\limits_{|z|=r} \frac{|a_n|}{|z-z_n|}>e^{n\epsilon}$ only for finitely many times. From here we get $|L_n(z)| \leq M_\epsilon + ne^{n\epsilon}$, where $M_\epsilon$ is a finite random number which is obtained by bounding the finite number of terms for which $\sup\limits_{|z|=r} \frac{|a_n|}{|z-z_n|}>e^{n\epsilon}$ is satisfied. Therefore we get that $\limsup\limits_{n \rightarrow \infty}\frac{1}{n}\log|L_n(z)|< \epsilon$ almost surely. 
\end{proof}
Notice that we have proved a stronger version of the Lemma \ref{momentthm2}. We will state this as a remark which will be used further lemmas.
\begin{remark}\label{one}
	Define $M_n(R):=\sup\limits_{|z|=R}|L_n(z)|$. Then for any $\epsilon>0$, we have $$\limsup\limits_{n \rightarrow \infty}\frac{1}{n}\log M_n(R)< \epsilon$$ for almost every $R>0$.
\end{remark}
For proving a similar result for the lower bound of $\log|L_n(z)|$ and establish \eqref{A2}, we need the Kolmogorov-Rogozin inequality  \ref{KR-Inequality} which was stated at the beginning of this chapter.

\begin{lemma}\label{krthm2}
	Let $L_n(z)$ be as in Theorem \ref{thm2}. Then for any $\epsilon>0$, and Lebesgue a.e. $z \in \mathbb{C}$ $$\lim\limits_{n \rightarrow \infty}\Pr\left(\frac{1}{n}\log|L_n(z)|<-\epsilon\right)=0.$$  
\end{lemma}

\begin{proof}
	Fix $z \in \mathbb{C}$ which is not in the exceptional set $F$. Let $z_{i_1},z_{i_2},\dots z_{i_{l_n}}$ be the points in $K_z$ from the set  $\{z_1,z_2,\dots,z_n\}$. From the definition of concentration function and the fact that the concentration function $ Q(X_1+X_2+\dots+X_n,\delta)$ is decreasing in $n$ we get,
	\begin{align}
	\Pr\left(|L_n(z)|\leq e^{-n\epsilon}\right) & \leq  Q\left(\sum_{i=1}^{n}\frac{a_i}{z-z_i},e^{-n\epsilon}\right), \\ 
	&\leq   Q\left(\sum_{k=1}^{l_n}\frac{a_{i_k}}{z-z_{i_k}},e^{-n\epsilon}\right). 
	\end{align} 
	The random variables $\frac{a_{i_k}}{z-z_{i_k}}$s are independent. Hence we can apply Kolmogorov-Rogozin inequality to get,
	\begin{align}
	\Pr\left(|L_n(z)|\leq e^{-n\epsilon}\right) & \leq  C_\epsilon\left\{\sum_{i=1}^{l_n}\left(1-Q\bigg(\frac{a_{i_k}}{z-z_{i_k}},e^{-n\epsilon}\bigg)\right)\right\}^{-\frac{1}{2}}.
	\end{align}
	Because $|z-z_{i_k}|\leq d(z,K_z)+diam(K_z)$, from above we get, 
	\begin{align}
	\Pr\left(|L_n(z)|\leq e^{-n\epsilon}\right) & \leq  C_\epsilon \left\{\sum_{i=1}^{l_n}\left(1-Q\left(a_{i_k},\left(d(z,K_z)+diam(K_z)\right)e^{-n\epsilon}\right)\right)  \right\}^{-\frac{1}{2}}\label{eqn:lemma5.4.8:1}
	\end{align}
	Because $a_{i_k}$s are non-degenerate i.i.d random variables and $l_n\rightarrow\infty$, the right hand side of \eqref{eqn:lemma5.4.8:1} converges to $0$ as $n\rightarrow\infty$. Hence the lemma is proved.
\end{proof}
It remains to show that the hypothesis of Theorem \ref{thm2} implies \eqref{A3}. Fix  $R>r$. The idea here is to write the function $\log|L_n(z)|$ for $z \in \text{$\mathbb{D}$}_r$ as an integral on the boundary of a larger disk $\text{$\mathbb{D}$}_R$ and bound the integral uniformly on the disk $\text{$\mathbb{D}$}_r.$ This is facilitated by Poisson-Jensen's formula for meromorphic functions.   
The Poisson-Jensen's formula is stated below. Let $\alpha_1, \alpha_2, \dots \alpha_k$ and $\beta_1,\beta_2, \dots \beta_\ell$ be the zeros and poles of a meromorphic function $f$ in $\text{$\mathbb{D}$}_R$. Then
\begin{align}	
\log|f(z)| = \frac{1}{2\pi}\int_{0}^{2\pi}\Re\bigg(\frac{Re^{i\theta}+z}{Re^{i\theta}-z}\bigg)\log|f(Re^{i\theta})|d\theta - \sum_{m=1}^{k}\log\bigg|\frac{R^{2}-\overline{\alpha}_jz}{R(z-\alpha_j)}\bigg|\\ + \sum_{m=1}^{l}\log\bigg|\frac{R^{2}-\overline{\beta}_jz}{R(z-\beta_j)}\bigg|
\end{align}


The following lemma \ref{tighttwo} gives an estimate of the boundary integral obtained in the Poisson-Jensen's formula when applied for the function $\log|L_n(z)|$ 	 at $z=0$. Define \[\text{$\mathcal{I}$}\text{$_n(z,R)$}:=\frac{1}{2\pi}\int_{0}^{2\pi}\Re\bigg(\dfrac{Re^{i\theta}+z}{Re^{i\theta}-z}\bigg)\log|L_n(Re^{i\theta})|d\theta.\] 
\begin{lemma}\label{tighttwo}
	There is a constant $c_2>0$ such that 
	\begin{align}
	\lim\limits_{n \rightarrow \infty}\Pr\left( \dfrac{1}{n}\text{$\mathcal{I}$}(0,R)\leq-c_2\right) = 0.
	\end{align}
\end{lemma}
\begin{proof}
	From Poisson-Jensen's formula at $0$ we get,
	\begin{equation}
	\dfrac{1}{n}\text{$\mathcal{I}$}\text{$_n(0;R) = \dfrac{1}{n}\log|L_n(0)| +\dfrac{1}{n}\sum_{m=1}^{k}\log\bigg|\dfrac{z_{i_m}}{R}\bigg| - \dfrac{1}{n}\sum_{m=1}^{l}\log\bigg|\dfrac{\alpha_{i_m}}{R}\bigg|$},\label{eqn:lemma5.4.9:1}
	\end{equation}
	
	where $z_{i_m}s$ and $\alpha_{i_m}s$ are zeros and critical points respectively of $P_n(z)$ in the disk $\text{$\mathbb{D}$}_R$. Because $0$ is not a limit point of $\{z_1,z_2, \dots\}$, $\left\{\dfrac{1}{n}\sum_{m=1}^{k}\log\big|\dfrac{z_{i_m}}{R}\big|\right\}_{n\geq1}$ is a sequence of negative numbers bounded from below. $\left\{\dfrac{1}{n}\sum_{m=1}^{l}\log\bigg|\dfrac{\alpha_{i_m}}{R}\bigg|\right\}_{n\geq1}$ is also a sequence of negative numbers. Therefore the last two terms in the right hand side of \eqref{eqn:lemma5.4.9:1} are bounded below. Because $0$ is not in exceptional set $F$, from Lemma \ref{krthm2} we have that the sequence $\lim\limits_{n\rightarrow\infty}\Pr\left(\frac{1}{n}\log|L_n(z)|<-1\right)=0$ is bounded from below. Therefore there exists $C_1$ such that $$\lim\limits_{n\rightarrow\infty}\Pr\left(\frac{1}{n}\log|L_n(z)|<-1 \text{ and }\dfrac{1}{n}\sum_{m=1}^{k}\log\bigg|\dfrac{z_{i_m}}{R}\bigg| - \dfrac{1}{n}\sum_{m=1}^{l}\log\bigg|\dfrac{\alpha_{i_m}}{R} <-C_1\right)=0.$$ Choosing $c_2=C_1+1$ the statement of lemma is established.
\end{proof}

Using above lemma \ref{tight} and exploiting formula of Poisson kernel for disk we will now obtain an uniform bound for the corresponding integral $\text{$\mathcal{I}$}_n(z,R)$.
\begin{lemma}\label{tight3}
	There is a constant $b>0$ such that for any $z \in \text{$\mathbb{D}$}\text{$_r$}$ 
	\begin{equation}
	\lim\limits_{n\rightarrow \infty}\Pr\left(\dfrac{1}{n}\text{$\mathcal{I}$}_n(z,R)\leq -b\right)=0.
	\end{equation}
\end{lemma}
\begin{proof}
	We will decompose the function $\log|L_n(z)|$ into its positive and negative components. Let $\log|L_n(z)|=\log_+|L_n(z)|-\log_-|L_n(z)|$, where $\log_+|L_n(z)|$ and $\log_-|L_n(z)|$ are positive. Using this we can write,
	
	\begin{align}
	2\pi\text{$\mathcal{I}$}\text{$_n(z)$} = & \int_{0}^{2\pi}\log|L_n(Re^{i\theta})|\Re\bigg(\dfrac{Re^{i\theta}+z}{Re^{i\theta}-z}\bigg)d\theta, \\
	= & \int_{0}^{2\pi}\log_+|L_n(Re^{i\theta})|\Re\bigg(\dfrac{Re^{i\theta}+z}{Re^{i\theta}-z}\bigg)d\theta -\int_{0}^{2\pi}\log_-|L_n(Re^{i\theta})|\Re\bigg(\dfrac{Re^{i\theta}+z}{Re^{i\theta}-z}\bigg)d\theta.
	\end{align}

	We can find constants $C_3$  and $C_4$ such that for any $z \in \text{$\mathbb{D}$}\text{$_r$}$,
	$0 < C_3 \leq \Re\bigg(\dfrac{Re^{i\theta}+z}{Re^{i\theta}-z}\bigg) \leq C_4 < \infty$ is satisfied. Therefore,
	\begin{align}
	2\pi\text{$\mathcal{I}$}\text{$_n(z)$} \geq & C_3\int_{0}^{2\pi}\log_+|L_n(Re^{i\theta})|d\theta - C_4\int_{0}^{2\pi}\log_-|L_n(Re^{i\theta})|^-d\theta,\\
	\geq & 2\pi C_3\text{$\mathcal{I}$}\textbf{$_n(0)-2\pi(C_4-C_3) M_n(R)$}.\label{eqn:lemma5.4.10:1}
	\end{align}
	From the Remark \ref{one} and Lemma \ref{tighttwo}  we get 
	\begin{equation}
	\lim\limits_{n \rightarrow \infty}\Pr\left( \dfrac{1}{n}\text{$\mathcal{I}$}\text{$_n(0)\leq -c \text{ or } \dfrac{1}{n}M_n(R) > 1 $}\right) = \textbf{$0$} \label{eqn:lemma5.4.10:2}
	\end{equation}
	
	The proof is completed from above \eqref{eqn:lemma5.4.10:2} and \eqref{eqn:lemma5.4.10:1} and by choosing $b=2\pi(cC_3+C_4-C_3)$.
\end{proof}

To complete the argument we now need to control the other terms in Poisson-Jensen's formula. It is shown in the forthcoming expressions. 

Let $\xi_{i_m}s$ and $\beta_{i_m}s$ be the poles and zeros of $L_n(z)$ in $\text{$\mathbb{D}$}_R$ and  $k,l(\leq n)$ are the number of zeros and poles of $L_n(z)$ respectively in $\text{$\mathbb{D}$}_R$. Now applying Poisson-Jensen's formula to $L_n(z)$ we have,

\begin{align}	
\frac{1}{n^{2}}\int_{\text{$\mathbb{D}$}\text{$_r$}}^{}&\log^{2}|L_n(z)|dm(z)   \\ &=\frac{1}{n^{2}}\int_{\text{$\mathbb{D}$}\text{$_r$}}\left(\text{$\mathcal{I}$}\textbf{$_n(z)+\sum\limits_{m=1}^{k}\log\biggl|\frac{R(z-\beta_{i_m})}{R^2-\overline{\beta}_{i_m}z}\biggr|+\sum\limits_{m=1}^{l}\log\biggl|\frac{R(z-\xi_{i_m})}{R^2-\overline{\xi}_{i_m}z}\biggr|$}\right)^2dm(z) 
\end{align}
Invoking a case of Cauchy-Schwarz inequality  $(a_1+a_2+\dots+a_n)^2\leq n(a_1^2+a_2^2+\dots+a_n^2)$ repeatedly we get,

\begin{align}
\int_{\text{$\mathbb{D}$}\text{$_r$}}^{}\dfrac{1}{n^{2}}&\log^{2}|L_n(z)|dm(z)
\\ & \leq  \dfrac{3}{n^{2}}\int_{\text{$\mathbb{D}$}\text{$_r$}}^{}|\text{$\mathcal{I}$}_n(z)|^2dm(z) + \dfrac{3}{n^{2}}\int_{\text{ $\mathbb{D}$}_r}^{}\left(\sum_{m=1}^{k}\log\bigg|\dfrac{R(z-\beta_{i_m})}{R^{2}-\overline{\beta}_{i_m}z}\bigg|\right)^{2}dm(z)\\ 
&+\dfrac{3}{n^{2}}\int_{\mathbb{D}\text{$_r$}}^{}\left(\sum_{m=1}\log\bigg|\dfrac{R(z-\xi_{i_m})}{R^{2}-\overline{\xi}_{i_m}z}\bigg|\right)^{2}dm(z),\\
&\leq  \int_{\mathbb{D}\text{$_r$}}^{}\dfrac{3}{n^{2}}|\text{$\mathcal{I}$}_n(z)|^2dm(z) + \dfrac{3k}{n^{2}}\sum_{m=1}^{k}\int_{\mathbb{D}\text{$_r$}}^{}\log^{2}\bigg|\dfrac{R(z-\beta_{i_m})}{R^{2}-\overline{\beta}_{i_m}z}\bigg|dm(z)\\
&+\dfrac{3l}{n^{2}}\sum_{m=1}^{l}\int_{\mathbb{D}\text{$_r$}}^{}\log^{2}\bigg|\dfrac{R(z-\xi_{i_m})}{R^{2}-\overline{\xi}_{i_m}z}\bigg|dm(z).
\end{align}
For $z \in\text{ $\mathbb{D}$}_r$, we have $|R^2-\overline{\beta}_{i_m}z|\geq R(R-r)$. Applying this inequality in the above we get,
\begin{align}
\int_{\mathbb{D}\text{$_r$}}^{}\dfrac{1}{n^{2}}\log^{2}|L_n(z)|dm(z)
& \leq  \int_{\mathbb{D}\text{$_r$}}^{}\dfrac{3}{n^{2}}|\text{$\mathcal{I}$}_n(z)|^2dm(z) + \dfrac{3k}{n^{2}}\sum_{m=1}^{k}\int_{\mathbb{D}\text{$_r$}}^{}\log^{2}\bigg|\dfrac{z-\beta_{i_m}}{R-r}\bigg|dm(z)\\
&+\dfrac{3l}{n^{2}}\sum_{m=1}^{l}\int_{\mathbb{D}\text{$_r$}}^{}\log^{2}\bigg|\dfrac{z-\xi_{i_m}}{R-r}\bigg|dm(z).  \label{eqn:poisson-jensen:1}
\end{align}

From the Lemmas \ref{tighttwo} and \ref{tight3}, the corresponding sequence $\frac{3}{n^{2}}\int_{\mathbb{D}\text{$_r$}}^{}|\text{$\mathcal{I}$}$$_n(z)|^2dm(z)$ is tight. The function $\log^{2}|z|$ is an integrable function on any bounded set in $\mathbb{C}$. Combining these facts and above inequality \eqref{eqn:poisson-jensen:1} we have that the sequences \linebreak $\left\{\int_{\mathbb{D}\text{$_r$}}^{}\frac{1}{n^{2}}\log^{2}|L_n(z)|dm(z)\right\}_{n\geq1}$ are tight. Hence the hypothesis of the Theorem \ref{thm2} implies \eqref{A3}. Therefore the proof of the theorem is complete.

\chapter{Matching between zeros and critical points of random polynomials}
\label{ch:matching6}
\section{Introduction}
In the previous chapters we have seen the behaviour of the point cloud of zeros and critical points in bulk. In this chapter we study the pairing of zeros and critical points. Dennis and Hannay in \cite{hannay}, gave an electrostatic argument to show that the zeros and critical points are closely paired for a generic (random) polynomial of higher degree. In \cite{hanin1} Hanin argued that the critical points and zeros for a random polynomial are mutually paired by computing the covariance between these measures. 


We will restrict our attention to critical points of polynomials having all real zeros. We choose the real zeros to be i.i.d. random variables. In the next section we show that the sum of distances between zeros and critical points, when paired appropriately, remains bounded under the assumption that the random variables have finite first moment. In the next section we study the spacings between the extremal zeros and critical points when the zeros are i.i.d. exponential or uniform random variables. We prove that the extremal critical point is much closer to the extremal zero than any other zeros.

%



\section{Matching distance between zeros and critical points of random polynomials.}

Matching between two sets is defined as follows. Let $U,V$ be two sets of finite and equal cardinality in the complex plane. A matching is  a bijection from $U$ to $V$. The concept of matching is used in qualitatively defining distance between two sets of same cardinality. Matching distance is a natural distance to quantify the closeness of the sets of zeros and critical points of a polynomial. There are several notions of matching distance. In this chapter we will deal with $\ell^1$ matching distance, which is defined as,

\[
d_1(U,V)=\inf\limits_{\pi\in \text{$\mathfrak{S}$}_n}\sum\limits_{i=1}^{n}|u_i-v_{\pi(i)}|,
\]
 where $U=\{u_1,u_2,\dots,u_n\}$, $V=\{v_1,v_2,\dots,v_n\}$ and $\pi=(\pi(1),\pi(2),\dots,\pi(n))$ is an element in set of permutations of size $n$ denoted by $\text{$\mathfrak{S}$}_n$.

To define mapping distance between set of zeros and critical points, we include the element $0$ in the set of critical points. We map the set of critical points to set of zeros so that the sum of the distances between the critical points and zero set is minimized and call that as matching distance. 

 The order statistics for a set of real numbers $\{\alpha_1,\alpha_2,\dots,\alpha_n\}$ are denoted as $\alpha_{(1)}\leq \alpha_{(2)} \leq \dots \leq \alpha_{(n)}$. In the following lemma we will compute the $\ell_1$ matching distance between two sets in the real line.

\begin{lemma}\label{matching_distance}
Let $X=\{x_1,x_2,\dots,x_n\}$ and $Y=\{y_1,y_2,\dots,y_n\}$ be two sets in the set of real numbers. Then the $\ell_1$ matching distance between $X$ and $Y$ is given by
\[
d_1(X,Y)=\sum\limits_{i=1}^{n}|x_{(i)}-y_{(i)}|.
\]
\end{lemma}
\begin{proof}
Without loss generality assume that $x_i=x_{(i)}$ and $y_i=y_{(i)}$ for all $i=1,2,\dots,n$. We will show that the matching distance is attained for identity matching. Suppose not, let $\pi$ be the permutation for which the matching distance is attained. Then there is $i<j$  such that $\pi(i)>\pi(j)$. If we tweak the permutation to $\pi'$ by choosing $\pi'(i)=\pi(j)$, $\pi'(j)=\pi(i)$ and $\pi'(\ell)=\pi(\ell)$ for $\ell\neq i,j$, then $\sum\limits_{i=1}^{n}|x_{i}-y_{\pi'(i)}|\leq\sum\limits_{i=1}^{n}|x_{i}-y_{\pi(i)}|$. Repeating this argument, it follows that $\sum\limits_{i=1}^{n}|x_{i}-y_{i}|\leq\sum\limits_{i=1}^{n}|x_{i}-y_{\pi(i)}|$. Hence the matching distance is attained for identity permutation.
\end{proof}
As an application the above Lemma \ref{matching_distance}, we compute the $\ell_1$ matching distance between the sets of zeros and critical points of a given polynomial. 
\begin{proposition}\label{matching_zeros_critical_positive}
 Let $x_1,x_2,\dots,x_n$ be all non-negative numbers. Then the $\ell_1$ matching distance between the sets of zeros and critical points of a polynomial $P_n(z)=(z-x_1)(z-x_2)\dots(z-x_n)$ is given by
\[
d_1(Z(P_n),Z(P_n')\cup\{0\})=\frac{1}{n}\sum\limits_{i=1}^{n}x_i.
\]
\end{proposition}
\begin{proof}
Let $\eta_1,\eta_2,\dots,\eta_{n-1}$ be the critical points of $P_n$. Because the critical points interlace the zeros of the $P_n$, we have $0\leq x_{(1)}\leq \eta_{(1)} \leq x_{(2)} \leq \dots, \leq \eta_{(n-1)} \leq x_{(n)}$. Applying Lemma \ref{matching_distance}, we get that 

\begin{equation}
d_1(Z(P_n),Z(P_n')\cup\{0\}) = \sum\limits_{i=2}^{n}(x_{(i)}-\eta_{(i-1)})+x_{(1)}=\sum\limits_{i=1}^{n}x_i-\sum\limits_{i=1}^{n-1}\eta_{i}.\label{l_1_dist_eqn}
\end{equation}

Recall Vieta's formula that if $\alpha_1,\alpha_2,\dots,\alpha_n$ are the roots of the polynomial defined as $P(z)=a_0+a_1z+\dots+a_nz^n$, then
	 \[ \sum\limits_{1\leq i_1<i_2<\dots<i_k\leq n}\alpha_{i_1}\alpha_{i_2}\dots \alpha_{i_k}=(-1)^k\frac{a_{n-k}}{a_n}.\]
	 From Vieta's formula, observe the identity $\sum\limits_{i=1}^{n}x_i = \dfrac{n}{n-1}\sum\limits_{i=1}^{n-1}\eta_i$.  Substituting this in \eqref{l_1_dist_eqn}, we get
	 \[
	 d_1(Z(P_n),Z(P_n')\cup\{0\}) =\sum\limits_{i=1}^{n}x_i-\frac{n-1}{n}\sum\limits_{i=1}^{n}x_i=\frac{1}{n}\sum\limits_{i=1}^{n}x_i.
	 \]
\end{proof}

\begin{proposition}\label{matching_zeros_critical}
Let $x_1,x_2,\dots,x_k$ be negative numbers and $x_{k+1},x_{k+2},\dots,x_{n}$ be non-negative numbers.
 Then the $\ell_1$ matching distance between the sets of zeros and critical points of a polynomial $P_n(z)=(z-x_1)(z-x_2)\dots(z-x_n)$ is bounded by
\[
d_1(Z(P_n),Z(P_n')\cup\{0\})\leq\frac{1}{k}\sum\limits_{i=1}^{k}|x_i|+\frac{1}{n-k}\sum\limits_{i=k+1}^{n}|x_i|.
\]
\end{proposition}
Before proving Proposition \ref{matching_zeros_critical}, we prove the following lemma, which indicates that the critical points move towards right when a new zero is introduced into the polynomial towards the left of all the zeros.
\begin{lemma}\label{assistant}
Let $\eta_1,\eta_2,\dots,\eta_{n-1}$ be the critical points of the polynomial $P(z)=(z-\alpha_1)(z-\alpha_2)\dots(z-\alpha_n)$. Let $\eta_0',\eta_1',\eta_2',\dots,\eta_{n-1}'$ be the critical points of $Q(z)=(z-\alpha)P(z)$ where $\alpha<\alpha_i$ for $i=1,2,\dots,n$. Then, $\alpha_{(i+1)}-\eta_{(i)}' \leq \alpha_{(i+1)}-\eta_{(i)}$, for $i=1,2,\dots,n$
\end{lemma}
\begin{proof}
It is enough to show that $\eta_{(i)}\leq\eta_{(i)}'$. Define $L_P(z):=\frac{P'(z)}{P(z)}=\sum\limits_{i=1}^{n}\frac{1}{z-\alpha_i}$ and $L_Q(z):=\frac{Q'(z)}{Q(z)}=\frac{1}{z-\alpha}+L_P(z)$. Both $L_P$ and $L_Q$ are decreasing functions in any interval which does not contain any of the zeros of $P$ and $Q$. Fix an $i\in \{1,2,\dots,n\}$. Because $\eta_{(i)}$ is a critical point of $P$, we have $L_P(\eta_{(i)})=0$ and $L_Q(\eta_{(i)})=\frac{1}{\eta_{(i)}-\alpha}>0$. But $L_Q$ vanishes exactly once in the interval $(\alpha_{(i)},\alpha_{(i+1)})$ at $\eta_{(i)}'$. Combing the facts that $L_Q$ is decreasing in $(\alpha_{(i)},\alpha_{(i+1)})$ and $L_Q(\eta_{(i)})>0$, we get that $\eta_{(i)}\leq\eta_{(i)}'$.
\end{proof}

\begin{proof}[Proof of Proposition \ref{matching_zeros_critical}]
With out loss of generality assume that $x_1\leq x_2\leq\dots x_k \leq0 \leq x_{k+1}\leq \dots x_n$. Let $\eta_1\leq\eta_2\dots \eta_{n-1}$ be the critical points of $P_n$. Factorize the polynomial $P_n$ as  $P_n(z)=Q_n(z)R_n(z)$, where $Q_n(z)=(z-x_1)(z-x_2)\dots(z-x_k)$ and $R_n(z)=(z-x_{k+1})(z-x_{k+2})\dots(z-x_{n})$. If $\eta_1'\leq\eta_2'\leq \dots \eta_{k-1}'$ and $\eta_{k+1}'\leq\eta_{k+2}'\leq\dots\eta_{n-1}'$ are critical points of $Q_n$ and $R_n$ respectively, then by repeatedly applying the previous Lemma \ref{assistant} to $Q_n$ and $R_n$, we get  $\eta_1-x_1 \leq \eta_1'-x_1, \dots, \eta_{k-1}-x_{k-1}\leq \eta_{k-1}'-x_{k-1}$ and $x_{k+2}-\eta_{k+1} \leq x_{k+2}-\eta_{k+1}', \dots,x_{n}-\eta_{n-1}\leq x_{n}-\eta_{n-1}'.$ The $\ell_1$ matching distances between zeros and critical points of $Q_n$ and $R_n$ are bounded by $\frac{1}{k}\sum\limits_{i=1}^{k}|x_i|$ and $\frac{1}{n-k}\sum\limits_{i=k+1}^{n}|x_i|$ respectively. Therefore we get,
\begin{align}
d_1(Z(P_n),Z(P_n')\cup\{0\}) &\leq d_1(Z(Q_n),Z(Q_n')\cup\{0\})+d_1(Z(R_n),Z(R_n')\cup\{0\}) \\&= \frac{1}{k}\sum\limits_{i=1}^{k}|x_i|+\frac{1}{n-k}\sum\limits_{i=k+1}^{n}|x_i|.
\end{align}
Hence the proposition is proved.
\end{proof}
Notice that Propositions \ref{matching_zeros_critical_positive} and \ref{matching_zeros_critical} are stated for deterministic polynomials. As an application we obtain the  matching distance for the polynomials whose zeros are i.i.d. random variables.

\begin{theorem}\label{random_matching_distance}
Let $X_1,X_2,\dots$ be i.i.d. random variables satisfying $\ee{\text{$|X_1|$}}$, define the polynomial $P_n(z)=(z-X_1)(z-X_2)\dots(z-X_n)$. Then 
\[
\limsup\limits_{n\rightarrow \infty}d_1(Z(P_n),Z(P_n')\cup\{0\})\leq \ee{\text{$|X_1|$}}.
\]
Moreover if $X_i$s are non-negative random variables, then
\[
\limsup\limits_{n\rightarrow \infty}d_1(Z(P_n),Z(P_n')\cup\{0\})= \ee{\text{$X_1$}}.
\]
\end{theorem}
Range of the zeros of $P_n$ gives a trivial bound for the $\ell_1$  matching distance between the zeros and critical points of $P_n$. If the random variables are all bounded then this $\ell_1$ matching distance remains bounded uniformly for any $n$. Where as if the random variables $X_i$s are unbounded the above Theorem \ref{random_matching_distance} shows that the $\ell_1$ matching distance between the zeros and critical points of $P_n$ remains bounded almost surely uniformly for any $n$.
\begin{proof}[Proof of Theorem \ref{random_matching_distance}]
The  second part of the theorem follows immediately by applying law of large numbers for Proposition \ref{matching_zeros_critical_positive}. For the first part, write $X_i=X_i^+-X_i^-$, where  $X_i^+\geq0$ and $X_i^-<0$. Let $k_n=\#\{X_i:X_i<0\}$. Applying Proposition \ref{matching_zeros_critical} we get that $$d_1(Z(P_n),Z(P_n')\cup\{0\})\leq\frac{1}{k_n}\sum\limits_{i=1}^{n}X_i^-+\frac{1}{n-k_n}\sum\limits_{i=1}^{n}X_i^+.$$ Applying law of large numbers for the above we get $$\limsup\limits_{n\rightarrow \infty}d_1(Z(P_n),Z(P_n')\cup\{0\}) \leq \ee{\text{$X_1^-$}}+\ee{\text{$X_1^+$}}=\ee{\text{$|X_1|$}}.$$
\end{proof}

\section{Spacings of zeros and critical points of random polynomials.}
Theorem \ref{random_matching_distance} shows that even if $X_i$s are unbounded random variables then the $\ell_1$ matching distance between the zeros and critical points remain bounded. This indicates that the extremal critical points stay much closer to one of the zeros than the others. We formalize this in the case of exponential random variables as the following result.  
\begin{theorem}\label{exponential}
	Let $X_1,X_2,\dots X_n$ be i.i.d exponential random variables. Let $\eta_{(1)} \leq \eta_{(2)} \leq \dots  \leq \eta_{(n-1)}$ be the critical points of the polynomial $P_n(z):=(z-X_1)(z-X_2)\dots(z-X_n)$. Then the following hold true,
	\begin{itemize}
		\item[1.]
		$n\log n(\eta_{(1)}-X_{(1)})\rightarrow 1$ in probability.
		\item[2.]
		$n\log n(X_{(n)}-\eta_{(n-1)})\rightarrow 1$ in probability. 
	\end{itemize}
	
\end{theorem}
We will use R\'{e}nyi's representation~\cite{boucheron} for the order statistics of exponential random variables while proving Theorem \ref{exponential} and is stated below. 

\paragraph{R\'{e}nyi's representation for order statistics:}
Let $Y_{(1)},Y_{(2)},\dots,Y_{(n)}$ be the order statistics of the sample of i.i.d exponential random variables, then
$$(Y_{(1)},Y_{(2)},\dots,Y_{(n)})   \,{\buildrel d \over =}\, \left(\frac{E_n}{n},\frac{E_{n-1}}{n-1}+\frac{E_n}{n},\dots,E_1+\frac{E_2}{2}+\dots+\frac{E_n}{n}\right),$$ where $E_1,E_2,\dots,E_n$ are i.i.d exponential random variables.
\begin{proof}[Proof of Theorem \ref{exponential}]
	
	Let $L_n(z):= \dfrac{P_n'(z)}{P_n(z)}=\sum\limits_{i=1}^{n}\frac{1}{z-X_i}=\sum\limits_{i=1}^{n}\frac{1}{z-X_{(i)}}$. If $z$ is a critical point of $P_n(z)$ then it satisfies $L_n(z)=0$.
	Hence from the equation $L_n(z)=0$ we have,
	\begin{align}
	\eta_{(1)}-X_{(1)} =& \left(\sum\limits_{i=2}^{n}\frac{1}{X_{(i)}-\eta_{(1)}}\right)^{-1} \\ \leq& \left(\sum\limits_{i=2}^{n}\frac{1}{X_{(i)}-X_{(1)}}\right)^{-1}
	\end{align}
	But $X_{(i)}-X_{(1)}=\frac{E_{n-1}}{n-1}+ \dots + \frac{E_{n-i+1}}{n-i+1}$ for all $i=2,3,\dots,n$, where $E_1,E_2,\dots,E_n$ are i.i.d exponential random variables. From R\'{e}nyi's representation it can be noticed that $\frac{E_{n-1}}{n-1}+ \dots + \frac{E_{n-i}}{n-i}     \,{\buildrel d \over =}\, Y_{(i)}$ for $i=1,2,\dots,n-1$, where $Y_{(1)},Y_{(2)},\dots,Y_{(n-1)}$ are order statistics of $Y_1,Y_2,\dots,Y_{n-1}$ which are i.i.d exponential random variables. Therefore,
	
	\begin{align}
		\eta_{(1)}-X_{(1)}  \,{\buildrel d \over =}\,&
	 \left(\sum\limits_{i=2}^{n}\frac{1}{\frac{E_{n-1}}{n-1}+ \dots + \frac{E_{n-i+1}}{n-i+1}}\right)^{-1}\\
	=& \left(\sum\limits_{i=1}^{n-1}\frac{1}{Y_{(i)}}\right)^{-1}
	= \left(\sum\limits_{i=1}^{n-1}\frac{1}{Y_i}\right)^{-1}.\label{eqn:thm1:1}
	\end{align}
	 Observe that $\frac{1}{Y_1}$ is regularly varying-1 and applying central limit theorem ( Chapter 2, Theorem 7.7 in \cite{durrett}) for \eqref{eqn:thm1:1} we get,
	\begin{equation}\label{dist}
	\frac{(\eta_{(1)}-X_{(1)})^{-1}-n\log n}{n} \overset{d}{\to} R
	\end{equation}
	where $R$ has stable-1 distribution. From \eqref{dist} it follows that $n\log n(\eta_{(1)}-X_{(1)})\overset{p}{\to}1.$
	
	The proof of the second statement $n\log n(X_{(n)}-\eta_{(n-1)})\stackrel{p}{\rightarrow} 1$  is similar as that of the first statement.
\end{proof}

\begin{remark}In the above Theorem \eqref{exponential}, instead of i.i.d exponential random variables we can choose i.i.d uniform random variables and obtain the same result. One may need to use the fact that if $U_1,U_2,\dots,U_n$ are i.i.d uniform random variables, then the order statistics
\begin{align}
 (U_{(1)},U_{(2)}&,\dots,U_{(n)}) \stackrel{d}{=}\\
 &\left(\frac{E_1}{E_1+E_2+\dots+E_{n+1}},\frac{E_1+E_2}{E_1+E_2+\dots+E_{n+1}}, \dots,\frac{E_1+E_2+\dots+E_n}{E_1+E_2+\dots+E_{n+1}}\right),
\end{align}

where $E_1,E_2,\dots,E_{n+1}$ are i.i.d exponential random variables.
\end{remark}

We believe the same result holds for any random variables whose densities satisfy certain regularity properties. The proof may be following the same idea but using R\'{e}nyi's representation theorem in more general form as given in \cite{boucheron} .

\chapter{Determinantal point processes from product of random matrices.}
\label{ch:ginibreproduct2}
\section{Introduction}
A Ginibre matrix is a random matrix whose entries are i.i.d. complex Gaussian random variables. In this chapter we derive exact eigenvalue density for certain products of random matrices. In this section we give a over view of the results where the exact eigenvalue density was obtained. Then we state our result and show that the earlier results were special cases of our result. In the next section we give a brief discussion about generalized Schur decomposition, which will used as a transformation to derive the eigenvalue density. In the later section we compute the Jacobian for this transformation. We complete the proof in the last section.
 
We will now recall a well known fact (Theorem 4.5.5 in \cite{manjubook}, Lemma 4 in \cite{soshnikovsurvey}) about the determinantal point process on the complex plane. Let the vector  $(z_1,z_2,\dots,z_n)$ be a random vector in $\text{$\mathbb{C}$}^n$ having density proportional to $\prod\limits_{i<j}|z_i-z_j|^2$ w.r.t a measure $\mu^{\otimes n}$. Notice that by doing column operations on the matrix $$\left[ \begin{smallmatrix}
\phi_0(z_1) & \phi_1(z_1) & \dots & \phi_{n-1}(z_1) \\
\phi_0(z_2) & \phi_1(z_2) & \dots & \phi_{n-1}(z_2) \\
\vdots & \vdots & \ddots & \vdots
\\
\phi_0(z_n) & \phi_1(z_n) & \dots & \phi_{n-1}(z_n) \end{smallmatrix}\right], $$ where $\phi_i$s are orthonormal polynomial w.r.t measure $\mu$, we get

\[
\det\left[ \begin{smallmatrix}
\phi_0(z_1) & \phi_1(z_1) & \dots & \phi_{n-1}(z_1) \\
\phi_0(z_2) & \phi_1(z_2) & \dots & \phi_{n-1}(z_2) \\
\vdots & \vdots & \ddots & \vdots
\\
\phi_0(z_n) & \phi_1(z_n) & \dots & \phi_{n-1}(z_n)\end{smallmatrix} \right] \left[ \begin{smallmatrix}
\phi_0(z_1) & \phi_1(z_1) & \dots & \phi_{n-1}(z_1) \\
\phi_0(z_2) & \phi_1(z_2) & \dots & \phi_{n-1}(z_2) \\
\vdots & \vdots & \ddots & \vdots
\\
\phi_0(z_n) & \phi_1(z_n) & \dots & \phi_{n-1}(z_n)\end{smallmatrix} \right]^* = c_n\prod\limits_{i<j}|z_i-z_j|^2,
\]
for some constant $c_n$. On simplification the left hand side of the above equation can be written as $\det[((\textbf{$\mathbb{K}$}(z_i,z_j)))_{1\leq i,j \leq n}]$, where $\textbf{$\mathbb{K}$}_n(z,w)=\sum\limits_{i=0}^{n-1}\phi_i(z)\overline{\phi}_i(w)$. Therefore the entries of the vector $(z_1,z_2,\dots,z_n)$ form a determinantal point process with kernel given by \[\textbf{$\mathbb{K}$}_n(z,w)=\sum\limits_{i=0}^{n-1}\phi_i(z)\overline{\phi}_i(w),\]
where $\phi_0, \phi_1, \dots, \phi_n$ are the orthonormal polynomials w.r.t measure $\mu$ on complex plane.

Ginibre \cite{ginibre} introduced three ensembles of matrices  with i.i.d. real, complex and quaternion  Gaussian entries respectively without imposing a Hermitian condition. These matrices are called Ginibre matrices in the literature. Here we restrict our attention to matrices with i.i.d. complex Gaussian entries. In \cite{ginibre}, Ginibre derived the eigenvalue density for $n\times n$ matrix with  i.i.d. standard complex Gaussian entries. 
\begin{theorem}[Ginibre ~\cite{ginibre}]
	Let $A$ be an $n\times n$ matrix with i.i.d standard complex Gaussian entries. Then the eigenvalues of $A$ form a determinantal point process on the complex plane with kernel
	\begin{equation*}
		\text{$\mathbb{K}$}_n(z,w) = \sum_{k=0}^{n-1}\frac{(z\bar w)^k}{k!}
	\end{equation*}
	w.r.t to background measure $\frac{1}{\pi}e^{-|z|^2}dm(z)$. Equivalently, the vector of eigenvalues has density
\[
	\frac{1}{\pi^n\prod_{k=1}^{n}k!}e^{-\sum_{k=1}^{n}|z_k|^2}\prod_{i<j}|z_i - z_j|^2 
\]
	w.r.t Lebesgue measure on $\text{$\mathbb{C}$}^n$.	
\end{theorem}
These are the first non-hermitian matrix ensembles for which the exact eigenvalue density is computed. 


 Later Krishnapur \cite{manjunath} showed that the eigenvalues of $A^{-1}B$ form a determinantal point process on the complex plane when $A$ and $B$ are independent random matrices with i.i.d. standard complex Gaussian entries. In random matrix literature this matrix ensemble $A^{-1}B$ is known as \textit{spherical ensemble}.
\begin{theorem}[M.Krishnapur ~\cite{manjunath}]
	Let $A$ and $B$ be i.i.d. $n\times n$ matrix with i.i.d. standard complex Gaussian entries. Then the eigenvalues of $A^{-1}B$ form a determinantal point process on the complex plane with kernel 
	
	\begin{equation*}
	\text{$\mathbb{K}$}_n(z,w) = (1+z\overline{w})^{n-1}
	\end{equation*}
	w.r.t to background measure $\frac{n}{\pi}\frac{dm(z)}{(1+|z|^2)^{n+1}}$. Equivalently, the vector of eigenvalues has density
	$$
	\frac{1}{n}\left(\dfrac{n}{\pi}\right)^n\prod_{k=1}^{n}{{n-1}\choose k}\prod_{k=1}^{n}\dfrac{1}{(1+|z_k|^2)^{n+1}}\prod_{i<j}|z_i - z_j|^2
	$$
	w.r.t Lebesgue measure on $\mathbb{C}\text{$^n$}$
\end{theorem}

 Akemann and Burda \cite{akemann} have derived the eigenvalue density for the product of $k$ independent $n\times n$ matrices with i.i.d. complex Gaussian entries.
 In this case the joint probability distribution of the eigenvalues of the product matrix is found to be given by a determinantal point process as in the case of Ginibre,
 but with a weight given by a Meijer $G$-function depending on $k$. Their derivation hinges on the generalized Schur decomposition for matrices and the method of orthogonal polynomials. We shall state that result as the following theorem.
\begin{theorem}[Akemann-Burda ~\cite{akemann}]
	Let $A_1,A_2,\dots,A_n$ be i.i.d. $n\times n$ matrices with i.i.d. standard complex Gaussian entries. Then the eigenvalues of $A_1A_2\dots A_n$ has density
	(with respect to Lebesgue measure on $\text{$\mathbb{C}$}^{n}$) proportional to
	\[
	\prod_{\ell=1}^{n}\omega(z_{\ell})\prod_{i<j}^{n}|z_i-z_j|^2
	\]
	with a weight function $ \omega(z)$, where
	\begin{equation}\label{weight0}
	\omega(z)=\int_{x_1\cdots x_{k}
	=z}e^{-\sum_{j=1}^{k}|x_{j}|^2}
	\prod_{j=1}^{k}|x_{j}|^{(n-1)}d\sigma.
	\end{equation}
	Here $\sigma$ is the Lebesgue measure restricted to the hyper surface $\{(x_1,x_2,\dots,x_k):x_1\cdots x_{k}
		=z\}$.
\end{theorem}

In all the above results after calculating the density of the eigenvalues, it turns out that they form a determinantal point process.

Now following the work of Krishnapur \cite{manjunath} on spherical ensembles and the work of Akemann and Burda \cite{akemann} on the product of $k$ independent  $n\times n$ Ginibre matrices, it is a natural question to ask, what can be said about the eigenvalues of  product of $k$ independent Ginibre matrices when a few of them are inverted? We investigated the case  $A=A_1^{\epsilon_1}A_2^{\epsilon_2}
\cdots A_k^{\epsilon_k}$, where each $\epsilon_i$
is $+1$ or $-1$ and $A_1,
A_2,\ldots, A_k$ are independent matrices with i.i.d. standard complex Gaussian entries, and obtained the density of eigenvalues of $A$. We state the result as the following theorem.

\begin{theorem}[K Adhikari, K Saha, NK Reddy, TR Reddy ~\cite{atr}]\label{chap5:thm1}
Let $A_1,A_2,\ldots,A_k$ be independent $n\times n$ random matrices with i.i.d. standard complex Gaussian entries.
Then the eigenvalues of  $A=A_1^{\epsilon_1}A_2^{\epsilon_2}
\ldots A_k^{\epsilon_k}$, where each $\epsilon_i$
is  $+1$ or $-1$, has density
(with respect to Lebesgue measure on $\mathbb{C}\textbf{$^n$}$) proportional to
\[
\prod_{\ell=1}^{n}\omega(z_{\ell})\prod_{i<j}^{n}|z_i-z_j|^2
\]
with a weight function $ \omega(z)$, where
\begin{equation}\label{weight1}
\omega(z)=\int_{x_1^{\epsilon_1}\cdots x_{k}^{\epsilon_k}
=z}e^{-\sum_{j=1}^{k}|x_{j}|^2}
\prod_{j=1}^{k}|x_{j}|^{(1-\epsilon_j)(n-1)}d\sigma.
\end{equation}

\no Here $\sigma$ is the Lebesgue measure restricted to the hyper surface given by  $\{(x_1,x_2,\dots,x_k):x_1^{\epsilon_1}\cdots x_{k}^{\epsilon_k}
	=z\}$.
\end{theorem}
\begin{remark}
From the symmetry of the expressions in the Theorem \ref{chap5:thm1} notice that the density of the eigenvalues of the matrix $A=A_1^{\epsilon_1}A_2^{\epsilon_2}
\ldots A_k^{\epsilon_k}$ depends only on  $\sum_{i=1}^{k}\epsilon_i$ but not on individual $\epsilon_i$s.
\end{remark}
\begin{remark}
  If $k=2$, $\epsilon_1=-1$ and $\epsilon_2=1$, then  from \eqref{weight1} we get that
\[\omega(z)=\int_{\frac{x_{2}} {x_{1}}=z}e^{- (|x_{1}|^{2}+|x_{2}|^2)}
|x_{1 }|^{2(n-1)}d\sigma=C_n \frac{1}{(1+|z|^2)^{(n+1)}},
\]
where $\sigma$ is the Lebesgue measure restricted to the hyper surface $\{(x_1,x_2):\frac{x_{2}} {x_{1}}=z\}$  $C_n$ is a constant . Hence the density of the eigenvalues of $A_1^{-1}A_2$
is proportional to

\[
\prod_{i=1}^{n}\frac{1}{(1+|z_i|^2)^{n+1}}\prod_{i<j}|z_{i}-z_j|^2.
\]
From the above expression it is clear that the eigenvalues of $A_1^{-1}A_2$ form a determinantal point process in a complex plane. This result was  proved by Krishnapur in \cite{manjunath} through a different approach.
\end{remark}

\begin{remark}
  If $\epsilon_i=1$ for $i=1,2,\ldots,k$, then by Theorem \ref{chap5:thm1} it follows that the eigenvalues of
$A_1A_2\ldots A_k$  form a determinantal point process. This result is due to Akemann and Burda \cite{akemann}.

\end{remark}
 For proving the Theorem \ref{chap5:thm1}, we will do an appropriate transformation and then integrate the auxiliary variables to obtain the eigenvalue density. We will use generalized Schur decomposition for the matrices $X_1,X_2,\dots,X_k$, where $X_i=A_i^{\epsilon_i}$ for $i=1,2,\dots,k$. In the forthcoming section we will present generalized Schur decomposition appeared in \cite{akemann}. In the subsequent section we compute the Jacobian for this transformation.

\section{Generalized Schur decomposition}
We will first recall the Schur decomposition and then state generalized Schur decomposition.
\paragraph{Schur decomposition:}Let $A$ be a $n\times n$ matrix. Then there exists an unitary matrix $U$, a diagonal matrix $D$ and a strictly upper triangular matrix $T$, such that $A=U(D+T)U^*$. The diagonal elements of $D$ are  the eigenvalues of $A$. Further, if the eigenvalues of $A$ are all distinct and we fix the order of their appearance in $D$, then the decomposition is unique up to a conjugation by a diagonal matrix, whose diagonal entries are in $S^1$.

The key ingredient in deriving the result \ref{chap5:thm1} is the generalization of the above mentioned Schur product. 

\paragraph{Generalized Schur decomposition:}Let $A_1,A_2,\dots,A_k$ be $n \times n$ matrices, then there exists unitary matrices $U_1,U_2,\dots,U_k$, diagonal matrices $D_1,D_2,\dots,D_k$ and strictly upper triangular matrices $T_1,T_2,\dots,T_k$ such that they can be decomposed in the following form.
\begin{align}
A_1 & = U_1(Z_1+T_1)U_2^*,\\
 A_2 & = U_2(Z_2+T_2)U_3^*,\\
 &\vdots\\
 A_k &=U_k(Z_k+T_k)U_1^*.
\end{align}

To prove this decomposition we first consider the Schur decomposition for the matrix $A_1A_2\dots A_k$. Let it be $A_1A_2\dots A_k=U_1(Z+T)U_1^*$. Now performing the Gram-Schimdt orthogonalization to the columns of the matrix $U_1^*A_1$, we get $U_1^*A_1=(Z_1+T_1)U_2^*$ for some unitary matrix $U_2$, and $Z_1, T_1$ being diagonal and strictly upper triangular matrices respectively. We repeat this process i.e., in the $i^{th}$ step we perform the Gram-Schimdt orthogonalization to the matrix $U_i^*A_i$ to get $U_i^*A_i=(Z_i+T_i)U_(i+1)$. After performing $n-1$ steps it forces that $U_n^*A_n=(Z_n+T_n)U_1^*.$  

This decomposition is not unique in general. But, if we incorporate certain conditions on the matrices then it will be unique. Assume that the diagonal entries of $Z_1Z_2\dots Z_k$ are distinct, and appear in a particular order (in particular may choose lexicographical ordering). Observe that replacing $U_i$ with $\Theta_iU_i$, where $\Theta_i$ is a diagonal unitary matrix, we have that $A_i$s assume a similar decomposition.
Hence if all the diagonal entries of $U_i$s are non-zero we may assume them to be positive.  On assuming these two conditions the decomposition will be unique. Another criterion for the uniqueness of this decomposition is to assume that the first non-zero entry in each row of $U_i$ is non-negative. In the next section while computing Jacobian we assume that all the diagonal entries are positive as the unitary matrices with zero diagonal entries form a null set. 

Notice that the eigenvalues of $A_1A_2\dots A_k$ are same as that of $Z_1Z_2\dots Z_k$. We will exploit this and use generalized Schur decomposition as the transformation in recovering the eigenvalue density.
We will compute the Jacobian for this transformation in the next section. For a more general discussion on this the reader can refer to the appendix in \cite{atr}.

\section{Jacobian computation}

To obtain eigenvalue density of $A$, we need to do an appropriate change of variables. We do generalized Schur decomposition as mentioned in the previous section. We will compute the Jacobian for this transformation in this section. The computation of Jacobian is on the lines of the computation, given in \cite{manjubook} (Section 6.3, Chapter-6), while deriving the eigenvalue density for Ginibre matrices. 

Before doing the Jacobian determinant calculation, we state a basic property about wedge product, which will be used repeatedly.

If $dy_j=\sum_{k=1}^na_{j,k}dx_k$, for $1\le j\le n$, then using the alternating property $dx\wedge dy=-dy\wedge dx$ it is easy to see that
\begin{equation}\label{eqn:wedge:relation}
dy_1\wedge dy_2\wedge \ldots\wedge dy_n=\det[((a_{j,k}))_{j,k\le n}]dx_1\wedge x_2\wedge\ldots\wedge dx_n.
\end{equation}

As a consequence we can see that,  if $\underline{x}=(x_1,\dots,x_n)$ is a unitary transformation of $\underline{y}=(y_1,\dots,y_n)$, then \begin{equation}\label{eqn:wedge:relation2}
dy_1\wedge d\overline{y}_1\wedge dy_2\wedge d\overline{y}_2 \ldots\wedge dy_n\wedge d\overline{y}_n=dx_1\wedge d\overline{x}_1\wedge dx_2\wedge d\overline{x}_2 \ldots\wedge dx_n\wedge d\overline{x}_n.
\end{equation}

%
From generalized Schur decomposition we have that for any matrices $X_1,X_2,\dots,X_k$ in $g\ell (n,\mathbb{C})$ can be written as
\begin{align}\label{eqn:gschur}
X_1 & = U_1(Z_1+T_1)U_2^*,\\
 X_2 & = U_2(Z_2+T_2)U_3^*,\\
 &\vdots\\
 X_k &=U_k(Z_k+T_k)U_{k+1}^*.
\end{align}
Where $U_1,U_2,\dots,U_k,U_{k+1}$ are unitary matrices satisfying $U_{k+1}=U_1$, $Z_1,Z_2,\dots,Z_k$ are diagonal matrices and $T_1,T_2,\dots,T_k$ are strictly upper triangular matrices. Because $U_iU_i^*=\text{$\mathbb{I}$}_n$  we have $(dU_i)U_i^*=-U_i(dU_i^*)$, for $i=1,2,\dots,k$. Using this fact and the generalized Schur decomposition \ref{eqn:gschur}, for any $\ell=1,2,\dots,k$ we get,

\begin{align}
dX_\ell &= (dU_\ell)(Z_\ell+T_\ell)U_{\ell+1}^*+U_\ell (dZ_\ell+dT_{\ell})U_{\ell+1}^*+U_\ell (Z_\ell+T_{\ell})dU_{\ell+1}^*\\
&= (dU_\ell)(Z_\ell+T_\ell)U_{\ell+1}^*+U_\ell (dZ_{\ell}+dT_\ell)U_{\ell+1}^*-U_\ell (Z_\ell+T_\ell)U_{\ell+1}^*(dU_{\ell+1})U_{\ell+1}^*\\
&=
U_\ell\left[(U_\ell^*dU_\ell)(Z_\ell+T_\ell)-(Z_\ell+T_\ell)(U_{\ell+1}^*dU_{\ell+1})+dZ_\ell+dT_\ell\right]U_{\ell+1}^*.
\end{align}

For convenience let us denote $\Lambda_\ell:=U_\ell^*(dX_\ell)U_{\ell+1}$, $\Omega_\ell:=U_\ell^*dU_\ell$ and $S_\ell:=Z_\ell+T_\ell$.  Note that $\Lambda_\ell=(\lambda_\ell (i,j))$ and $\Omega_\ell=(\omega_\ell (i,j))$ are $n\times n$ matrices of one forms, and $dS_\ell$ ($=dZ_\ell+dT_\ell$) is an upper triangular matrix of one form. Let  $Z_\ell=\mbox{diag}(Z_\ell(1),\linebreak Z_\ell(2),\dots,Z_\ell(n))$, $T_\ell=(t_\ell(i,j))$ and $\Lambda_\ell=(\lambda_\ell(i,j))$.
Define,
\begin{equation}\label{eqn:lambda_l}
\Lambda_\ell=\Omega_\ell S_\ell-S_\ell\Omega_{\ell+1}+dS_\ell.
\end{equation}
For any unitary matrix $U$ the transformation $X \rightarrow UX$ is a unitary transformation. Therefore from \eqref{eqn:wedge:relation2}, we have
\[\bigwedge_{i,j}\left(dX_\ell(i,j)\wedge d\overline{X}_\ell(i,j)\right)= 
\bigwedge_{i,j}\left(d\lambda_\ell(i,j)\wedge d\overline{\lambda}_\ell(i,j)\right),\]
for $\ell=1,2,\dots,k.$
Throughout this computation we will ignore the constants, hence each equality is indeed an equality up to a constant. Because we will be dealing with probability densities the constants can be retrieved by equating the integral to $1$.
Expanding the equation \eqref{eqn:lambda_l} we get,
\begin{eqnarray}\label{eqn:wedge_nullify}
\lambda_\ell (i,j) &=&\sum\limits_{m=1}^{j}S_\ell (m,j)\omega_\ell (i,m)-\sum\limits_{m=i}^{n}S_\ell (i,m)\omega_{\ell+1}(m,j)+dS_\ell (i,j)
\\&=&\left\{
\begin{array}{lcr}
S_\ell (j,j)\omega_\ell (i,j)-S_\ell (i,i)\omega_{\ell+1}(i,j) \\ +\left[\sum\limits_{m=1}^{j-1}S_\ell (m,j)\omega_\ell (i,m)-\sum\limits_{m=i+1}^{n}S_\ell (i,m)\omega_{\ell+1}(m,j)\right] \mbox{  if } i>j;\\
dS_\ell (i,j)+S_\ell (i,j)\left(\omega_\ell (i,i)-\omega_{\ell+1}(j,j)\right)\\ +\left[\sum\limits_{\substack{m=1\\m\neq i}}^{j}S_\ell (m,j)\omega_\ell (i,m)-\sum\limits_{\substack{m=i+1\\m\neq j}}^{n}S_\ell (i,m)\omega_{\ell+1}(m,j)\right] 
\mbox{ if }i\leq j.\\
\end{array}
\right.
\end{eqnarray}

To execute the wedge product, we will now arrange $\{\lambda_\ell (i,j),\overline{\lambda}_\ell (i,j)\}$ in a particular order. We will use lexicographic order on the indices associated with $\lambda_\ell (i,j)$. The indices associated with $\lambda_\ell (i,j)$ are $(i,j,\ell)$. The lexicographical order will be taken on $(i,n-j,k-\ell)$. The corresponding conjugate terms will be followed by the term for which it is conjugate. For convenience, we present the ordering as the following table (each row is read from left to right and top row precedes bottom rows). 
\[
\begin{smallmatrix}
\lambda_1(n,1), \overline{\lambda}_1(n,1), \dots, \lambda_k(n,1), \overline{\lambda}_k(n,1),&\dots&, \lambda_1(n,n), \overline{\lambda}_1(n,n), \dots, \lambda_k(n,n), \overline{\lambda}_k(n,n)\\
\lambda_1(n-1,1), \overline{\lambda}_1(n-1,1), \dots, \lambda_k(n-1,1), \overline{\lambda}_k(n-1,1),&\dots&, \lambda_1(n-1,n), \overline{\lambda}_1(n-1,n), \dots, \lambda_k(n-1,n),\overline{\lambda}_k(n-1,n) \\
\vdots &\ddots&\vdots\\
\lambda_1(1,1), \overline{\lambda}_1(1,1), \dots, \lambda_k(1,1), \overline{\lambda}_k(1,1),&\dots&, \lambda_1(1,n), \overline{\lambda}_1(1,n), \dots, \lambda_k(1,n),\overline{\lambda}_k(1,n)
\end{smallmatrix}
\]
 Using the fact that $\Omega_\ell$ is skew-hermitian (i.e., $\omega_\ell(i,j)=-\overline{\omega}_\ell(j,i)$), while executing the wedge product, notice that the terms in the square brackets are one forms and have already appeared before in the given ordering. Hence their contribution to the entire product is nullified. In the next couple of paragraphs we will explain in detail about this cancellation, the reader who already got convinced may skip them.
 
 If $i>j$, then each of the terms in the square brackets contain either $\omega_\ell(i,m_1)$ or $\omega_{\ell+1}(m_2,j)$, where $m_1<j$  and $m_2>i$. These terms have already appeared in $\lambda_\ell(i,m_1)$ and $\lambda_{\ell+1}(m_2,j)$ respectively, outside the square brackets, which are leading the order we have executed the product.
 
 If $i\leq j$, then each of the terms in the square brackets contain either $\omega_\ell(i,m_1)$ or $\omega_{\ell+1}(m_2,j)$, where $m_1\leq j$, $m_1 \neq i$, $m_2>i$ and $m_2\neq j$. For the case $j\geq m_1>i$, by skew hermitian property of $\Omega_\ell$, we have $\omega_\ell(i,m_1)=-\overline{\omega}_\ell(m_1,i)$ which has already appeared in $\overline{\lambda}_\ell(m_1,i)$, outside the square brackets. For the case $m_1<i$, $\omega_\ell(i,m_1)$ has already appeared in $\lambda_\ell(i,m_1)$ outside square brackets. Similarly for the case $m_2<j$,we have $\omega_{\ell+1}(m_2,j)=-\overline{\omega}_{\ell+1}(j,m_2)$ which has appeared outside square brackets in $\overline{\lambda}_{\ell+1}(j,m_2)$. Lastly in the case of $m_2>j$, $\omega_{\ell+1}(m_2,j)$ has already appeared in $\lambda_{\ell+1}(m_2,j).$ 
 
Therefore, if we assume 
\begin{eqnarray}
\mu_{\ell}(i,j) &=&\left\{
\begin{array}{lcr}
S_\ell (j,j)\omega_\ell (i,j)-S_\ell (i,i)\omega_{\ell+1} (i,j) & \mbox{  if } i>j;\\
dS_\ell (i,j)+S_\ell (i,j)\left(\omega_\ell (i,i)-\omega_{\ell+1} (j,j)\right) & \mbox{ if }i\leq j;
\end{array}
\right.
\end{eqnarray}

then,
\[\bigwedge_{\ell}^{}\bigwedge_{i,j}^{}\left(\lambda_\ell (i,j)\wedge \overline{\lambda}_\ell (i,j)\right)=\bigwedge_{\ell}^{}\bigwedge_{i,j}^{}\left(\mu_\ell (i,j)\wedge \overline{\mu}_\ell (i,j)\right).\]
Recall that $S_\ell=Z_\ell+T_\ell$. For $i>j$, the term $\bigwedge_\ell(\mu_\ell(i,j)\wedge\overline{\mu}_\ell(i,j))$ yields $\big|\prod\limits_{\ell=1}^{k}Z_\ell(i)-\prod\limits_{\ell=1}^{k}Z_\ell(j)\big|^2$. 
 Hence we get,
\begin{align}\label{eqn:manifold_nullification_1}
\bigwedge_{\ell}\bigwedge_{i,j}|\lambda_\ell(i,j)|^2 =\left(\prod_{i>j}\bigg|\prod\limits_{\ell=1}^{k}Z_\ell(i)-\prod\limits_{\ell=1}^{k}Z_\ell(j)\bigg|^2\right)&\bigwedge_{\ell}\left(\bigwedge_{i>j}|\omega_\ell(i,j)|^2\bigwedge_{i}|dZ_\ell(i)|^2\right)\\\times\bigwedge_{\ell}
\bigwedge_{i,j}\big|dT_\ell(i,j)&+T_\ell(i,j)\left(\omega_\ell(i,i)-\omega_{\ell+1}(j,j)\right)\big|^2.
\end{align}
Note that we have simplified the notation by denoting $|dz|^2:=dz\wedge d\overline{z}$. Consider the set of unitary matrices $\text{$\mathcal{M}$}_\ell=\{U_\ell:U_\ell^*U_\ell=\text{$\mathbb{I}$}_n,U_\ell(i,i)>0\}$. $\text{$\mathcal{M}$}_\ell$ is a sub manifold of dimension $n^2-n$ in the manifold  $\mathcal{U}\textbf{$(n)$}$. The dimension of $\text{$\mathcal{M}$}_\ell$ can be obtainded by imposing the constrains $\{U_\ell(i,i)>0;i=1,2,\dots,n\}$, on $\mathcal{U}\textbf{$(n)$}$. Now the product $\omega_\ell(m,m)\wedge\left(\bigwedge_{i>j}|\omega_\ell(i,j)|^2\right)=0$, because $\omega_\ell(i,j)$s are one forms on the manifold $\text{$\mathcal{M}$}_\ell$ whose dimension is $n^2-n$ and the product contains $n^2-n+1$ terms. Hence \eqref{eqn:manifold_nullification_1}, will be reduced to,
\begin{align}
\bigwedge_{\ell}\bigwedge_{i,j}&|\lambda_\ell(i,j)|^2\\&=\left(\prod_{i>j}\bigg|\prod\limits_{\ell=1}^{k}Z_\ell(i)-\prod\limits_{\ell=1}^{k}Z_\ell(j)\bigg|^2\right)\bigwedge_{\ell}\bigwedge_{i>j}|\omega_\ell(i,j)|^2\bigwedge_{i}\bigwedge_{\ell}|dZ_\ell(i)|^2\bigwedge_{\ell}\bigwedge_{i<j}|dT_\ell(i,j)|^2\label{eqn:manifold_nullification_2}.
\end{align}

Notice that $\bigwedge_{i>j}|\omega_\ell(i,j)|^2$ is $n^2-n$ form on space of unitary matrices whose diagonal entries are non-negative (or in other words $\text{$\mathcal{U}$}(n)/\text{$\mathcal{U}$}(1)$) and it is invariant under any unitary transformation. Hence the measure induced by this form is Haar measure on $\text{$\mathcal{U}$}(n)/\text{$\mathcal{U}$}(1)$.  We will denote it by $|dH(U_\ell)|$. Therefore we have 

\begin{align}
\bigwedge_{\ell}\bigwedge_{i,j}|\lambda_\ell(i,j)|^2=&\left(\prod_{i>j}\bigg|\prod\limits_{\ell=1}^{k}Z_\ell(i)-\prod\limits_{\ell=1}^{k}Z_\ell(j)\bigg|^2\right)\\ &\times\bigwedge_{\ell}|dH(U_\ell)|\bigwedge_{i}\bigwedge_{\ell}|dZ_\ell(i)|^2\bigwedge_{\ell}\bigwedge_{i<j}|dT_\ell(i,j)|^2\label{eqn:manifold_nullification_3}.
\end{align}

 Now that we have the basic ingredients ready, we will proceed for the proof of the theorem in the next section.

\section{Proof of Theorem \ref{chap5:thm1}}
The density of $(A_1,A_2,\ldots,A_k)$ is proportional to
\[
\prod_{\ell=1}^{k}e^{-\Tr( A_{\ell}A_{\ell}^*)}\bigwedge_{\ell=1}^{k}\bigwedge_{i,j=1}^n|dA_{\ell}(i,j)|^2
\]
where $|dA_{\ell}(i,j)|^2=dA_{\ell}(i,j)\wedge d\bar{A}_{\ell}({i,j}).$ Through out the proof, we will ignore the proportionality constants where ever present. Since we are dealing with probability densities, the proportionality  constants can be recovered by equating the integral of the density to $1$.

Let $X_{\ell}=A_{\ell}^{\epsilon_\ell}$ for $\ell=1,2,\ldots,k$. The Jacobian for the transformation $A_\ell \rightarrow A_\ell^{\epsilon_\ell}$ is $|\det(A_\ell)|^{2(\epsilon_\ell-1)n}$. Hence  the  joint density of $(X_1,X_2,\ldots,X_k)$ is
proportional to
\begin{align}
\prod_{\ell=1}^{k}e^{-\Tr( X_{\ell}^{\epsilon_{\ell}}X_{\ell}^{\epsilon_{\ell}*})}\prod_{\ell=1}^k
|\det(X_{\ell})|^{2(\epsilon_{\ell}-1)n}
\bigwedge_{\ell=1}^{k}\bigwedge_{i,j=1}^n|dX_{\ell}(i,j)|^2 \label{eqn:dencityofx}
\end{align}
Using  generalized Schur decomposition we have
\begin{align}
X_{\ell}=U_{\ell}S_{\ell}U_{\ell+1}^*,\;\;\mbox{for $\ell=1,2,\ldots,k,$} \label{eqn:schurdecomposition}
\end{align}
where $U_{k+1}=U_1$ and $U_{\ell}$ are unitary matrices, $S_{\ell}$ are upper triangular matrices. We write  $S_{\ell}$ for $1\leq l \leq k$ as
\begin{equation}
S_{\ell}=Z_{\ell}+T_{\ell}, \label{eqn:diag+uppertriangular}
\end {equation}
where $Z_{\ell}=\mbox{diag}(Z_{\ell }(1),Z_{\ell }(2),\ldots,Z_{\ell}(n))$ and $T_{\ell}$ are strictly upper triangular matrices.
Now from  \eqref{eqn:manifold_nullification_3},
we have
\begin{eqnarray}
\bigwedge_{\ell=1}^{k}\bigwedge_{i,j=1}^n|dX_{\ell}(i,j)|^2&=&\prod_{i<j}|z_i-z_j|^2\bigwedge_{\ell=1}^k\big(|dH(U_\ell)|\bigwedge_{i\le j}|dS_{\ell}(i,j)|^2\big)\\ \nonumber
&=&\prod_{i<j}|z_i-z_j|^2\bigwedge_{\ell=1}^k\bigg(|dH(U_\ell)|\bigwedge_{i<j}|dT_{\ell}(i,j)|^2\bigwedge_{i=1}^n |dZ_\ell(i)|^2\bigg),\label{eqn:{wedge}}
\end{eqnarray}
where
 $|dH(U_{\ell})|$ are independent Haar measures on $\mathcal{U}\text{$(n)$}/\mathcal{U}\text{$(1)$}$ and $z_i=\prod _{\ell=1}^k Z_{\ell}(i)$ for $1\leq i \leq n$. Notice that $z_1,z_2,\ldots,z_n$
 are the eigenvalues of $X_1X_2\cdots X_k$. Now using \eqref{eqn:schurdecomposition}
  and \eqref{eqn:{wedge}},  \eqref{eqn:dencityofx} can be written as
\begin{align}
&\prod_{\ell=1}^{k}\left[e^{-\Tr( S_{\ell}^{\epsilon_{\ell}}S_{\ell}^{\epsilon_{\ell}*})}
 |\det(S_{\ell})|^{2(\epsilon_{\ell}-1)n}\right]\prod_{i<j}|z_i-z_j|^2\\ &\times \bigwedge_{\ell=1}^k\bigg(|dH(U_\ell)|
\bigwedge_{i<j}|dT_{\ell}(i,j)|^2\bigwedge_{i=1}^n |dZ_{\ell}(i)|^2\bigg).\label{eqn:sipmle}
\end{align}
Now our aim is to integrate out all auxiliary variables from \eqref{eqn:sipmle} to get the density of eigenvalues of $X_1X_2\cdots X_k$. 
Observe that if $S=Z+T$ where $Z=diag(x_1,x_2,\ldots,x_n)$ and $T$ is a strictly upper triangular matrix, then
 \[  S^{-1}S^{-1*}=(I+Z^{-1}T)^{-1}Z^{-1}Z^{-1*}(I+Z^{-1}T)^{-1*}.\]
Observe that $Z^{-1}T$ is also a strictly upper triangular matrix. If we let  $P=Z^{-1}T$, we get
$P$ is an upper triangular matrix and
\[
|DP|=\prod_{i=1}^{n-1}\frac{1}{|x_{i}|^{2(n-i)}}|DT|,
\]
where $|DP|=\bigwedge_{i,j}|dP(i,j)|^2$ and $|DT|=\bigwedge_{i,j}|dT(i,j)|^2$.
Now replacing $(I+P)^{-1}$ by $Q$, we have $|DQ|=|DP|$ and therefore
\[
|DT|=\prod_{i=1}^{n-1}|x_{i}|^{2(n-i)}|DQ|.
\]
Hence we have
\begin{align}
e^{-\Tr(S^{-1}S^{-1*})}|DT|&= e^{-\Tr(QZ^{-1}Z^{-1*}Q^*)}\prod_{i=1}^{n-1}|x_{i}|^{2(n-i)}|DQ|\\ 
&= e^{-\sum_{i=1}^{n}\frac{1}{|x_{i}|^2}(\sum_{j=1}^{i-1}|Q(j,i)|^2+1)}\prod_{i=1}^{n-1}|x_{i}|^{2(n-i)}|DQ|.\label{eqn:simplification}
\end{align}
Now using  \eqref{eqn:simplification} for each $S_{\ell}$, $1\leq \ell \leq k$, we get that the expression \eqref{eqn:sipmle} is proportional to
\begin{eqnarray}
\prod_{i<j}|z_i-z_j|^2\prod_{\ell=1}^{k}\bigg[e^{-\sum_{i=1}^n(|Z_{\ell}(i)|^{2\epsilon_{\ell}}
+\frac{1-\epsilon_{\ell}}{2}|Z_{\ell }(i)|^{2\epsilon_\ell}\sum_{j=1}^{i-1}|Q_{\ell}(j,i)|^2
+\frac{1+\epsilon_{\ell}}{2}\sum_{j=i+1}^n|T_\ell(i,j)|^2)}
\\\times\prod_{i=1}^{n-1}|Z_{\ell}(i)|^{(n-i)(1-\epsilon_\ell)}
\prod_{i=1}^{n}|Z_{\ell }(i)|^{2(\epsilon_{\ell}-1)n}\bigg]\bigwedge_{\ell=1}^k|dH({U_{\ell}})|
|DQ_{\ell}|^{\frac{1-\epsilon_{\ell}}{2}}|DT_{\ell}|^{\frac{1+\epsilon_{\ell}}{2}}|DZ_{\ell}|.\nonumber
\end{eqnarray}

Now integrating this expression with respect to all variables except the variables of $Z_{\ell}$ for $\ell=1,2,\ldots,k$, we have
(omitting the constant)
\begin{eqnarray}\label{eqn:result1}
\prod_{i<j}|z_i-z_j|^2\prod_{\ell=1}^{k}\left(e^{-\sum_{i=1}^n(|Z_{\ell }(i)|^{2\epsilon_{\ell}})}
\prod_{i=1}^{n}|Z_{ \ell }(i)|^{(\epsilon_{\ell}-1)(n+1)}\right)\bigwedge_{\ell=1}^k |DZ_{\ell}|.
\end{eqnarray}
Now integrating the above expression on the respective hypersurfaces given by \linebreak  $\left\{(Z_{1}(i),Z_{2}(i),\dots, Z_{k}(i)):Z_{1}(i)Z_{2}(i)\dots Z_{k}(i)=z_i\right\}$, for $i=1,2,\dots,n$
 we get the  density of the eigenvalues $(z_1,z_2,\ldots,z_n)$ of
$A_1^{\epsilon_1}A_2^{\epsilon_2}\cdots A_k^{\epsilon_k}$.
This completes the proof of the theorem. \qed

\chapter{Probability that products of real random matrices have all eigenvalues real tend to $1$.}
\label{ch:realeigenvalues3}
\section{Background}
In this chapter we consider products of real random matrices with fixed size. 
In \cite{arul}, Arul Lakshminarayan observed an interesting phenomenon in products of Ginibre matrices. He considered products of i.i.d Ginibre matrices with real Gaussian entries. Let $p_n^{(k)}$ be the probability that product of $n$ such matrices have all real eigenvalues. Using numerical simulations he computed $p_n^{(k)}$. Based on the observations, he conjectured that $p^{(k)}_n$ increases to $1$ with the size of the product.



 Peter Forrester, in \cite{forrester}, considered the case of $k \times k$ Ginibre matrices with real Gaussian entries. He gave a formula for $p_n^{(k)}$. From that formula he deduced that this probability increases to $1$ exponentially.

We state a generalization of the conjecture stated in \cite{arul}.
\begin{conjecture}\label{con1}Let $X_1,X_2, \dots X_n$ be i.i.d. matrices of size $k \times k$, whose entries are i.i.d. real random variables distributed according to probability measure $\mu$ and $A_n=X_1X_2\dots X_n$. Then,
	$$\lim\limits_{n\rightarrow \infty}\Pr(A_n\text{ has all real eigenvalues})=1.$$
\end{conjecture}

\section{Results}
We prove the conjecture for the special case when the probability measure $\mu$ has an atom i.e., there is $x \in \mathbb{R}$ such that $\mu(\{x\})>0$. The proof for this case is based on a simple observation that rank of product of matrices is at most the minimum of the ranks of the individual matrices. In the given scenario, each individual matrix will be of rank at most $1$ with non zero probability. If a real matrix has rank at most $1$, then it has all real eigenvalues (they are $0$ and the trace of the matrix). The following theorem formalizes the result.

\begin{theorem} 
Let $X_1,X_2, \dots X_n$ be i.i.d. matrices of size $k \times k$, whose entries are i.i.d. real random variables distributed according to an atomic probability measure $\mu$ and $A_n=X_1X_2\dots X_n$.  Then, 
\[ 
\lim\limits_{n\rightarrow\infty}\Pr(A_n \text{ has all real real eigenvalues})=1.
\]

\end{theorem}
\begin{proof}
Let $x$ be an atom of measure $\mu$. Then $X_j$ has rank at most $1$, with probability at least $\mu(\{x\})^{k^2}$. All the matrices $X_j$s are independent of each other. Therefore,
\begin{align}
\Pr(A_n \text{ has rank at most } 1 ) & \geq \Pr(\text{at least one of }X_1,X_2,\dots,X_n \text{ has rank at most }1),\\
& \geq 1-(1-\mu(\{x\})^{k^2})^n.\label{eqn:chapter4:lemma1:1}
\end{align}
 
We know that real matrices with rank at most $1$, have all eigenvalues real. Hence,
$$
\Pr (A_n \text{ has all real eigenvalues}) \geq \Pr(A_n \text{ has rank at most }1).
$$ 	
Hence from above and \eqref{eqn:chapter4:lemma1:1}	we have,  $\lim\limits_{n\rightarrow\infty}\Pr(A_n \text{ has all real eigenvalues})=1.$
\end{proof}

We make the following observation about $2 \times 2$ real matrices. It says that if the rows or columns of a real random matrix are exchangeable then that matrix has both real eigenvalues with probability at least $\frac{1}{2}$. Later this will be applied to product of $2 \times 2$ real Ginibre matrices, which gives us the probability that the product of $2 \times 2$ real Ginibre matrices is at least $\frac{1}{2}$.

\begin{lemma}\label{halfbound}
Let $M = \left[\begin{smallmatrix} a&b\\ c&d \end{smallmatrix}\right]$, where $(a,b)$ and $(c,d)$ are real exchangeable random variables. Then,
 $$
\Pr(M \text{ has both real eigenvalues}) \geq \frac{1}{2}.
$$
\end{lemma}
\begin{proof}
	
	The characteristic polynomial of the matrix $M$ is $P_M(x)=x^2-(a+d)x+(ad-bc)$. The matrix $M$ has all real eigenvalues if the discriminant of the characteristic polynomial,
	 $$(a+d)^2-4(ad-bc) \geq 0.$$
	 Hence the probability that $M$ has both real eigen values is,
	 $$\Pr((a+d)^2-4(ad-bc) \geq 0).$$
	 Because $(a,c)$ and $(b,d)$ are exchangeable we have,
	 $$
	 \Pr((a+d)^2-4(ad-bc) \geq 0)=\Pr((b+c)^2-4(bc-bd) \geq 0).
	 $$ 
	 Therefore,
	 \begin{align}
	 \Pr(M \text{ has both} &\text{ real eigenvalues})\\ &= \frac{1}{2}(\Pr((a+d)^2-4(ad-bc) \geq 0)+\Pr((b+c)^2-4(bc-bd) \geq 0)),\\
	 & \geq \frac{1}{2}\Pr((a+d)^2-4(ad-bc) \geq 0 \text{ or }(b+c)^2-4(bc-bd) \geq 0). \label{eqn:chapter4:lemma2:1}
	 \end{align} 
	 Because $(a+d)^2-4(ad-bc)+(b+c)^2-4(bc-bd)\geq 0$, at least one of $(a+d)^2-4(ad-bc)$ and $(b+c)^2-4(bc-bd)$ is non-negative. Therefore,	 
	 $$
	 \Pr((a+d)^2-4(ad-bc) \geq 0 \text{ or }(b+c)^2-4(bc-bd) \geq 0)=1.
	 $$
	  Combining above and \eqref{eqn:chapter4:lemma2:1} we have that,
	  $\Pr(M \text{ has both real eigenvalues}) \geq \frac{1}{2}.$
	 \end{proof}
	  

Using Lemma \ref{halfbound}, in the case of $2 \times 2$ Ginibre matrices we can obtain that the probability of the products having all real eigenvalues to be at least $\frac{1}{2}$. 
\begin{corollary}
Let $X_1,X_2, \dots X_n$ are i.i.d. matrices of size $2 \times 2$ whose entries are i.i.d real random variables distributed according to probability measure $\mu$ and $A_n=X_1X_2\dots X_n$. Then,

$$
\Pr(A_n\text{ has all real eigenvalues})\geq \frac{1}{2}.
$$
\end{corollary}
\begin{proof}
To satisfy the hypothesis of Lemma \ref{halfbound}, it is enough to show that the rows of the matrix $A_n$ are exchangeable. It can be noticed that the matrices $X_1$ and $\left[\begin{smallmatrix}
0 & 1\\
1 & 0
\end{smallmatrix}\right]X_1$ are identically distributed, so are $A_n$ and $\left[\begin{smallmatrix}
0 & 1\\
1 & 0
\end{smallmatrix}\right]A_n$. Hence the rows of $A_n$ are exchangeable.\end{proof}

\backmatter
\appendix
\bibliographystyle{amsalpha} \bibliography{Thesis}

\newcommand{\etalchar}[1]{$^{#1}$}
\providecommand{\bysame}{\leavevmode\hbox to3em{\hrulefill}\thinspace}
\providecommand{\MR}{\relax\ifhmode\unskip\space\fi MR }
\providecommand{\MRhref}[2]{%
  \href{http://www.ams.org/mathscinet-getitem?mr=#1}{#2}
}
\providecommand{\href}[2]{#2}
\begin{thebibliography}{HKPV09}

\bibitem[AB12]{akemann}
G.~Akemann and Z.~Burda, \emph{Universal microscopic correlation functions for
  products of independent ginibre matrices}, Journal of Physics A: Mathematical
  and Theoretical \textbf{45} (2012), no.~46, 465201.

\bibitem[ARRS13]{atr}
K.~Adhikari, N.K. Reddy, T.R. Reddy, and K.~Saha, \emph{Determinantal point
  processes in the plane from products of random matrices}, arXiv preprint
  arXiv:1308.6817 (2013).

\bibitem[BT12]{boucheron}
S.~Boucheron and M.~Thomas, \emph{Concentration inequalities for order
  statistics}, Electron. Commun. Probab. \textbf{17} (2012), no. 51, 1--12.

\bibitem[CNTY14]{cheung}
P.~L. Cheung, T.~W. Ng, J.~Tsai, and S.~C.~P. Yam, \emph{Higher-order, polar
  and {S}z.-{N}agy's generalized derivatives of random polynomials with
  independent and identically distributed zeros on the unit circle},
  Computational Methods and Function Theory (2014), 1--28.

\bibitem[DH03]{hannay}
M.~R. Dennis and J.~H. Hannay, \emph{Saddle points in the chaotic analytic
  function and ginibre characteristic polynomial}, Journal of Physics A:
  Mathematical and General \textbf{36} (2003), no.~12, 3379.

\bibitem[Dur05]{durrett}
Rick Durrett, \emph{Probability: Theory and examples}, Duxbury advanced series,
  Thomson Brooks/Cole, 2005.

\bibitem[Ess68]{KR1}
Carl~G. Esseen, \emph{On the concentration function of a sum of independent
  random variables}, Zeitschrift f\"ur Wahrscheinlichkeitstheorie und Verwandte
  Gebiete \textbf{9} (1968), no.~4, 290--308 (English).

\bibitem[ET50]{erdos-turan}
P.~Erd\"{o}s and P.~Tur\'{a}n, \emph{On the distribution of roots of
  polynomials}, Annals of mathematics (1950), 105--119.

\bibitem[For14]{forrester}
Peter~J. Forrester, \emph{Probability of all eigenvalues real for products of
  standard gaussian matrices}, Journal of Physics A: Mathematical and
  Theoretical \textbf{47} (2014), no.~6, 065202.

\bibitem[Gin65]{ginibre}
Jean Ginibre, \emph{Statistical ensembles of complex, quaternion, and real
  matrices}, Journal of Mathematical Physics \textbf{6} (1965), no.~3,
  440--449.

\bibitem[Han15]{hanin1}
Boris Hanin, \emph{Correlations and pairing between zeros and critical points
  of gaussian random polynomials}, International Mathematics Research Notices
  \textbf{2015} (2015), no.~2, 381--421.

\bibitem[HKPV09]{manjubook}
J.B. Hough, M.~Krishnapur, Y.~Peres, and B.~Vir\'{a}g, \emph{Zeros of gaussian
  analytic functions and determinantal point processes}, University lecture
  series, American Mathematical Society, 2009.

\bibitem[Kab15]{kabluchko}
Zakhar Kabluchko, \emph{Critical points of random polynomials with independent
  identically distributed roots}, Proceedings of the American Mathematical
  Society \textbf{143} (2015), no.~2, 695--702.

\bibitem[KPP{\etalchar{+}}11]{borcea}
D.~Khavinson, R.~Pereira, M.~Putinar, E.~B. Saff, and S.~Shimorin,
  \emph{Borcea's variance conjectures on the critical points of polynomials},
  283--309. \MR{3051172}

\bibitem[Kri06]{manjunath}
Manjunath Krishnapur, \emph{Zeros of random analytic functions}, Ph.D. thesis,
  2006.

\bibitem[Lak13]{arul}
Arul Lakshminarayan, \emph{On the number of real eigenvalues of products of
  random matrices and an application to quantum entanglement}, Journal of
  Physics A: Mathematical and Theoretical \textbf{46} (2013), no.~15, 152003.

\bibitem[Mar66]{marden}
Morris Marden, \emph{Geometry of polynomials}, American Mathematical Society
  Mathematical Surveys, American Mathematical Society, 1966.

\bibitem[Mar83]{mardenconjectures}
\bysame, \emph{Conjectures on the critical points of a polynomial}, Amer. Math.
  Monthly \textbf{90} (1983), no.~4, 267--276. \MR{700266 (84e:30007)}

\bibitem[O'R14]{orourke}
Sean O'Rourke, \emph{Critical points of random polynomials and characteristic
  polynomials of random matrices}, arXiv preprint arXiv:1412.4703 (2014).

\bibitem[PR13]{pemantle}
Robin Pemantle and Igor Rivin, \emph{The distribution of zeros of the
  derivative of a random polynomial}, Advances in Combinatorics, Springer,
  2013, pp.~259--273.

\bibitem[Ran95]{ransford}
Thomas Ransford, \emph{Potential theory in the complex plane}, London
  Mathematical Society Student Texts, Cambridge University Press, 1995.

\bibitem[RS02]{rahmanbook}
Q.I. Rahman and G.~Schmeisser, \emph{Analytic theory of polynomials}, London
  Mathematical Society monographs, Clarendon Press, 2002.

\bibitem[Sos00]{soshnikovsurvey}
Alexander Soshnikov, \emph{Determinantal random point fields}, Russian
  Mathematical Surveys \textbf{55} (2000), no.~5, 923.

\bibitem[Sub12]{sneha}
Sneha~D. Subramanian, \emph{On the distribution of critical points of a
  polynomial}, Electronic Communications in Probability \textbf{17} (2012),
  no.~37.

\bibitem[Sur09]{sury}
B.~Sury, \emph{{When are complex zeroes on a Roll(e)?}}, Resonance \textbf{14}
  (2009), 872--881.

\bibitem[TV10]{taovu}
T.~Tao and V.~Vu, \emph{Random matrices: universality of {ESD}s and the
  circular law}, Ann. Probab. \textbf{38} (2010), no.~5, 2023--2065, With an
  appendix by Manjunath Krishnapur. \MR{2722794 (2011e:60017)}

\bibitem[Wal50]{walsh}
Joseph~L. Walsh, \emph{The location of critical points of analytic and harmonic
  functions}, American Mathematical Society: Colloquium publications, American
  Mathematical Society, 1950.

\end{thebibliography}

\end{document}